\documentclass{amsart}
\usepackage[utf8]{inputenc}
\usepackage{graphicx}
\usepackage{tikz-cd}
\usepackage{mathrsfs}
\usepackage{amsfonts}
\usepackage{amsthm}
\usepackage{amsmath}
\usepackage{amssymb}
\usepackage{mathtools}
\usepackage{extarrows}
\usepackage{interval}
\usepackage{bbm}
\usepackage[english]{babel}
\usepackage{romannum}
\usepackage{setspace}
\usepackage{listings}
\usepackage{pdfpages}
\usepackage{hyperref}
\usepackage{enumitem}

\newtheorem{theorem}{Theorem}[section]
\newtheorem{lemma}[theorem]{Lemma}
\newtheorem{corollary}[theorem]{Corollary}
\newtheorem{proposition}[theorem]{Proposition}

\theoremstyle{definition}
\newtheorem{definition}[theorem]{Definition}

\newtheorem{example}[theorem]{Example}
\numberwithin{equation}{section}
\newcommand*{\sheafhom}{\mathcal{H}\kern -.5pt om}
\newcommand*{\sheafext}{\mathcal{E}\kern -.5pt xt}

\title{Equivalences of derived categories of sheaves on tame stacks}
\author{Fei Peng}
\email{pengf2@student.unimelb.edu.au}
\address{School of Mathematics \& Statistics, The University of Melbourne, Parkville, VIC, 3010, Australia}
\date{May 30, 2024}

\begin{document}

\begin{abstract}
    Building on Olander's work on algebraic spaces, we prove Orlov's representability theorem relating fully faithful functors and Fourier--Mukai transforms between the bounded derived category of coherent sheaves to the case of smooth, proper, and tame algebraic stacks. This extends previous results of Kawamata for Deligne-Mumford stacks with generically trivial stabilizers and projective coarse moduli spaces.
\end{abstract}

\maketitle
\allowdisplaybreaks
\pagenumbering{arabic}

\section{Introduction}\label{1}

Let $X$ and $Y$ be two smooth projective varieties over a field $k$. Let $F\colon D^{b}_{coh}(X)\to D^{b}_{coh}(Y)$ be an exact $k$-linear functor. We say that $F$ is a {\it Fourier--Mukai transform} if there exists a bounded complex of coherent sheaves $K$ on $X\times_{k}Y$ such that $F\cong\Phi_{K}^{X\to Y}$, where $\Phi_{K}^{X\to Y}$ is given by$$Rp_{2,*}(Lp_{1}^{*}(-)\otimes_{\mathcal{O}_{X\times_{k}Y}}^{\mathbb{L}}K).$$We say $K$ is a {\it Fourier--Mukai kernel} of $F$.

A fundamental theorem of Orlov \cite[Theorem\ 2.2]{orlov97} states that if $F$ is a fully faithful functor that admits both a left and a right adjoint, then it is a Fourier--Mukai transform. Moreover, its Fourier--Mukai kernel is unique up to isomorphism. Note that the assumption about adjoints is now known to be superfluous by \cite[Theorem\ 1.3]{BvdB} and Serre duality. See \cite[Remark\ 2.1]{CS07} for the details.

Because of its importance and elegance, multiple attempts have been made to relax the assumptions in Orlov's theorem. Orlov and Lunts established a variant for perfect complexes on projective varieties using techniques from dg-categories \cite{orlov-lunts}. Around the same time, Ballard independently discovered this result in his thesis \cite{Ballard09}. His method can also be used to establish several related but weaker results for quasi-projective schemes. In terms of stacks, the only result in this direction was due to Kawamata, who established the theorem for smooth separated irreducible Deligne-Mumford stacks over the complex numbers with generically trivial stabilizers and projective coarse moduli spaces \cite[Theorem\ 1.1]{kawamata}. Using a natural combination of the methods in \cite{kawamata} and \cite{Ballard09}, one could generalize Kawamata's theorem to all smooth separated tame irreducible Deligne-Mumford stacks over a field with projective coarse moduli spaces. Note that all algebraic stacks with finite diagonal in characteristic 0 are tame and Deligne--Mumford. See \cite[Definition\ 3.1]{tamestacks08}. It is also possible to obtain a twisted variant of Orlov's theorem using the approach in \cite{CS07} in the above situation. Recently, Olander proved Orlov's theorem for smooth proper varieties in \cite{Ola24}, and smooth proper algebraic spaces over a field in his thesis \cite{olander_2022}. His method is interesting because it does not use an ample line bundle or the resolution property, which is essential in all the existing works.

The purpose of this note is to establish Orlov's theorem for any smooth proper tame algebraic stack over a field. We expect this to be folklore to experts but were unable to find it in the literature. More precisely, we prove the following result.
\begin{theorem}\label{1.1}
Let $\mathcal{X},\mathcal{Y}$ be two smooth proper tame algebraic stacks over a field $k$. If $F\colon D^{b}_{coh}(\mathcal{X})\to D^{b}_{coh}(\mathcal{Y})$ is a $k$-linear exact fully faithful functor, then it is a Fourier--Mukai transform. Moreover, the Fourier--Mukai kernel of $F$ is unique up to quasi-isomorphism.
\end{theorem}

One could also interpret Theorem \ref{1.1} from a dg-categorical perspective. Let $D_{dg}(\mathcal{X})$ and $D_{dg}(\mathcal{Y})$ be the dg-enhancements for $D_{coh}^{b}(\mathcal{X})$ and $D_{coh}^{b}(\mathcal{Y})$ respectively. Note that these dg-enhancements are unique for $\mathcal{X}$ and $\mathcal{Y}$ by \cite[Proposition\ 6.10]{CS18}. It is well-known that any dg-functor from $D_{dg}(\mathcal{X})$ to $D_{dg}(\mathcal{Y})$ is a Fourier--Mukai transform \cite{BZFN10}. In this context, Theorem \ref{1.1} states that any fully faithful functor from $D_{coh}^{b}(\mathcal{X})$ to $D_{coh}^{b}(\mathcal{Y})$ lifts to a fully faithful dg-functor between the corresponding dg-enhancements. As a consequence, we see that for any smooth proper tame algebraic stack $\mathcal{X}$ over a field, the dg-category $D_{dg}(\mathcal{X})$ is completely determined by the derived category $D_{coh}^{b}(\mathcal{X})$.
 
Our argument for Theorem \ref{1.1} is a natural extension of Olander's method for the case of algebraic spaces. A critical step in Olander's argument is to show that every $k$-linear exact equivalence between the categories of coherent sheaves on $\mathcal{X}$ and $\mathcal{Y}$ is a Fourier--Mukai transform. For algebraic spaces, this is an immediate consequence of Gabriel's reconstruction theorem \cite{gabriel}. In this case, any $k$-linear exact equivalence $F\colon\operatorname{Coh}(\mathcal{X})\to\operatorname{Coh}(\mathcal{Y})$ is naturally isomorphic to $f^{*}(-)\otimes\mathcal{L}$, where $f\colon\mathcal{Y}\to\mathcal{X}$ is an isomorphism of algebraic spaces, and $\mathcal{L}$ is a line bundle on $\mathcal{X}$.

Unfortunately, Gabriel's theorem is false for stacks. It is impossible to arrange for any $k$-linear exact equivalence $F\colon\operatorname{Coh}(\mathcal{X})\to\operatorname{Coh}(\mathcal{Y})$ to be the pullback along an isomorphism of algebraic spaces twisted by a line bundle. Consider the classifying stack $B(\mathbb{Z}/3\mathbb{Z})$ over the complex numbers. There is an equivalence of categories between $\operatorname{Coh}(B(\mathbb{Z}/3\mathbb{Z}))$ and the category $\operatorname{Rep}(\mathbb{Z}/3\mathbb{Z})$ of finite dimensional complex representation of $G$. In this case, line bundles on $B(\mathbb{Z}/3\mathbb{Z})$ are the irreducible representations of $\mathbb{Z}/3\mathbb{Z}$. Let $F\colon\operatorname{Coh}(B(\mathbb{Z}/3\mathbb{Z}))\to\operatorname{Coh}(B(\mathbb{Z}/3\mathbb{Z}))$ be the functor generated by exchanging the two non-trivial irreducible representations. It is clear that $F$ is an exact equivalence since $\mathbb{Z}/3\mathbb{Z}$ is linearly reductive. However, it is clear that $F$ does not come from pulling back along an automorphism of $B(\mathbb{Z}/3\mathbb{Z})$ twisted by a line bundle. To remedy the situation, we appeal to a different characterization of $k$-linear exact functors between $\operatorname{Coh}(\mathcal{X})$ and $\operatorname{Coh}(\mathcal{Y})$ using coherent sheaves on $\mathcal{X}\times_{k}\mathcal{Y}$ from the stacks project.

\begin{theorem}[cf.\ {\cite[\href{https://stacks.math.columbia.edu/tag/0FZN}{Tag 0FZN}]{stacks-project}}]\label{1.2}
Let $\mathcal{X},\mathcal{Y}$ be algebraic stacks of finite type over a field $k$. If $\mathcal{X}$ is tame with finite diagonal, then there exists an equivalence of categories between
\begin{enumerate}[topsep=0pt,noitemsep,label=\normalfont(\arabic*)]
    \item\label{1.3(1)} the category of $k$-linear exact functors $F\colon\operatorname{Coh}(\mathcal{X})\to\operatorname{Coh}(\mathcal{Y})$, and
    \item\label{1.3(2)} the full subcategory of $\operatorname{Coh}(\mathcal{X}\times_{k}\mathcal{Y})$ consisting of $\mathcal{K}$ that is flat over $\mathcal{X}$ and has support proper and quasi-finite over $\mathcal{Y}$
\end{enumerate}
sending a coherent sheaf $\mathcal{K}$ on $\mathcal{X}\times_{k}\mathcal{Y}$ to the exact functor $p_{2,*}(p_{1}^{*}(-)\otimes\mathcal{K})$.
\end{theorem}

Theorem \ref{1.2} tells us that we can always realize all the $k$-linear exact functor in Theorem \ref{1.2}\ref{1.3(1)}, not just equivalences, as the restrictions of Fourier--Mukai transforms. One might ask if we are able to refine Theorem \ref{1.2} to obtain a variant of Gabriel's theorem where the functors in Theorem \ref{1.2}\ref{1.3(1)} are given by a morphism from $\mathcal{Y}$ to $\mathcal{X}$. Unfortunately, this is only possible if these functors are monoidal and preserve line bundles. This is a consequence of Tannaka duality for algebraic stacks \cite[Theorem\ 1.1]{HR19}. On the other hand, Theorem \ref{1.2} fails if $\mathcal{X}$ has infinite stabilizers. In this case, one could still find a Fourier-Mukai kernel for any $k$-linear exact functor $F\colon\operatorname{Coh}(\mathcal{X})\to\operatorname{Coh}(\mathcal{Y})$. However, it might not be a coherent sheaf on $\mathcal{X}\times_{k}\mathcal{Y}$. See Example \ref{3.6} for a counterexample.

Our main results indicate that many classical results on derived categories of sheaves on schemes could possibly be extended to algebraic stacks. As an application of our main result, we show that the Krull dimension is a derived invariant for smooth, proper tame algebraic stacks.

\begin{theorem}[cf. {\cite[Theorem\ 12]{Ola23}}]\label{1.3}
Let $\mathcal{X},\mathcal{Y}$ be two smooth proper, and tame algebraic stacks over a field $k$. Let $F\colon D^{b}_{coh}(\mathcal{X})\to D^{b}_{coh}(\mathcal{Y})$ be a $k$-linear exact functor. If $F$ is fully faithful, then $\dim(X)\leq\dim(Y)$. If $F$ is an equivalence, then $\dim(X)=\dim(Y)$.
\end{theorem}

Our argument for Theorem \ref{1.3} follows closely from Olander's method for schemes and algebraic spaces. In particular, it uses the notion of the countable Rouquier dimension. We also established some basic facts about the countable Rouquier dimension for stacks, which might be of independent interest.

There are some other interesting possible applications of Theorem \ref{1.1}. For example, it would be interesting to establish variants of \cite[Theorem\ 7.1\ and\ 7.2]{kawamata} in positive characteristics to study the irrational geometry of varieties with tame quotient singularities. On the other hand, it is well known that for any smooth projective variety $X$, there are at most countably many, up to isomorphism, smooth projective varieties that are derived equivalent to $X$ \cite{AT09}. It would be interesting to extend this to smooth, proper, tame stacks.

This article is organized as follows. In section \ref{2}, we produce an infinite left resolution of the structure sheaf of the diagonal by coherent sheaves for Noetherian algebraic stacks with quasi-finite diagonal. In section \ref{3}, we prove Theorem \ref{1.2}. In section \ref{4}, we show that there exists a complex $K$ in $D_{coh}^{b}(\mathcal{X}\times_{k}\mathcal{Y})$ that agrees with our fully faithful functor $F$ on coherent sheaves on $\mathcal{X}$ with zero dimensional support. In section \ref{5}, we recall the notion of generalized points first introduced by Lim and Polishchuk \cite{LP21} and later generalized by Hall and Priver \cite{HP24}. We then use it to establish a few auxiliary results for our main theorem. In section \ref{6} we establish Theorem \ref{1.1}. Finally, we provide an application of Theorem \ref{1.1} on fully faithful functors and dimensions in section \ref{7}.

\subsection*{Acknowledgements}

This project is supported by the Melbourne Research Scholarship offered by the University of Melbourne. I am very grateful to my supervisor, Jack Hall, for many enlightening discussions and his guidance, support, and encouragement throughout this project. I would also like to thank Noah Olander for his comments and suggestions about an earlier draft of this article. 

\section{Resolution of the diagonal}\label{2}

In this section, we resolve the structure sheaf of the diagonal for certain algebraic stacks using coherent sheaves. The construction in this section is due to Olander for algebraic spaces \cite[Section\ 1.3]{olander_2022}, which we follow closely. Let $\mathcal{X}$ be a quasi-compact and quasi-separated algebraic stack with quasi-finite and separated diagonal. By \cite[Theorem\ 7.2]{quasifinite}, there exists a quasi-finite flat presentation $\pi\colon U\to\mathcal{X}$ where $U$ is an affine scheme and $\pi$ is separated. By Zariski's main theorem \cite[Theorem\ 8.6]{ZMT}, there exists a factorization $U\overset{j}{\longrightarrow}\bar{U}\overset{\bar{\pi}}{\longrightarrow}\mathcal{X}$, where $j$ is an open immersion and $\bar{\pi}$ is finite. Let $\mathcal{I}$ be a coherent sheaf of ideals corresponding to the complement of $U$ in $\bar{U}$. Set $\mathcal{F}_{n}=\bar{\pi}_{*}\mathcal{I}^{n}$ for all $n\in\mathbb{Z}_{\geq 0}$. This gives us a system of coherent sheaves$$\cdots\hookrightarrow\mathcal{F}_{n}\to\mathcal{F}_{n-1}\hookrightarrow\cdots\hookrightarrow\mathcal{F}_{1}\hookrightarrow\mathcal{F}_{0}$$on $\mathcal{X}$. We have the following variant of Deligne's formula.

\begin{lemma}[{\cite[Lemma\ 1.3.1]{olander_2022}}]\label{2.1}
There exists a system of compatible morphisms $\phi_{n}\colon\mathcal{O}_{U}\to L\pi^{*}\mathcal{F}_{n}$ such that for any object $\mathcal{K}$ in $D_{qc}(\mathcal{X})$, the canonical map$$\operatorname{colim}_{n}\operatorname{Hom}_{\mathcal{O}_{\mathcal{X}}}(\mathcal{F}_{n},\mathcal{K})\to\operatorname{Hom}_{\mathcal{O}_{U}}(\mathcal{O}_{U},L\pi^{*}\mathcal{K})$$sending $\varphi\colon\mathcal{F}_{n}\to\mathcal{K}$ to $L\pi^{*}(\varphi)\circ\phi_{n}$ is an isomorphism.
\end{lemma}

\begin{lemma}[{\cite[Lemma\ 1.3.2]{olander_2022}}]\label{2.2}
Let $\mathcal{K}$ be a complex in $D_{qc}(\mathcal{X})$. Then there exists a morphism$$\bigoplus_{i\in I}\mathcal{F}_{n_{i}}[d_{i}]\to \mathcal{K}$$ whose induced map on the cohomology sheaves is surjective. If $\mathcal{K}$ is concentrated in degree 0, then we may take $d_{i}=0$ for all $i\in I$. If $\mathcal{K}$ is in $D_{coh}^{b}(\mathcal{X})$, we may take the index set $I$ to be finite.
\end{lemma}

Using Lemma \ref{2.2}, we establish the following proposition for complexes on $\mathcal{X}\times_{k}\mathcal{X}$.

\begin{proposition}[{\cite[Lemma\ 1.3.3]{olander_2022}}]\label{2.3}
Let $R$ be a ring. Let $\mathcal{X}$ be a quasi-compact quasi-separated algebraic stack over $R$ with quasi-finite and separated diagonal. Let $\mathcal{K}$ be a complex in $D_{qc}(\mathcal{X}\times_{R}\mathcal{X})$. Then there exists a morphism$$\bigoplus_{i\in I}\mathcal{F}_{m_{i}}\boxtimes\mathcal{F}_{n_{i}}[d_{i}]\to \mathcal{K}$$ whose induced map on the cohomology sheaves is surjective. If $\mathcal{K}$ is concentrated in degree 0, then we may take $d_{i}=0$ for all $i\in I$. If $\mathcal{K}$ is in $D_{coh}^{b}(\mathcal{X})$, we may take the index set $I$ to be finite.
\end{proposition}

An immediate corollary of Proposition \ref{2.3} is the following resolution for coherent sheaves on $\mathcal{X}\times_{k}\mathcal{X}$.

\begin{corollary}\label{2.4}
Let $R$ be a ring. Let $\mathcal{X}$ be a quasi-compact and quasi-separated algebraic stack over $R$ with quasi-finite and separated diagonal. If $\mathcal{F}$ is a coherent sheaf on $\mathcal{X}\times_{R}\mathcal{X}$, then there exists a surjective morphism$$\mathcal{E}\boxtimes\mathcal{G}\to\mathcal{F},$$where $\mathcal{E}$ and $\mathcal{G}$ are coherent sheaves on $\mathcal{X}$.
\end{corollary}

We conclude this section with the following technical lemma, which we will use in the proof of our main theorem.

\begin{lemma}[cf.\ {\cite[Lemma\ 1.3.4]{olander_2022}}]\label{2.5}
Let $\mathcal{X}$ be an integral and proper algebraic stack over a field $k$ with finite diagonal. Let $\mathcal{G}$ be a coherent sheaf on $\mathcal{X}$. If $\mathcal{X}$ has strictly positive dimension, then$$\operatorname{Hom}_{\mathcal{O}_{\mathcal{X}}}(\mathcal{G},\mathcal{F}_{n})=0$$for any $n\gg 0$.
\end{lemma}

\begin{proof}
By assumption, we know that $U$ is reduced and the dimension of the local ring at closed points of $U$ is at least 1. This tells us that $U$ has depth at least 1 at every closed point. By construction, $$\operatorname{Hom}_{\mathcal{O}_{\mathcal{X}}}(\mathcal{G},\mathcal{F}_{n})=\operatorname{Hom}_{\mathcal{O}_{\bar{U}}}(\bar{\pi}^{*}\mathcal{G},\mathcal{I}^{n}).$$Since $\mathcal{X}$ is proper, we know that $\{\operatorname{Hom}_{\mathcal{O}_{\bar{U}}}(\bar{\pi}^{*}\mathcal{G},\mathcal{I}^{n})\}_{n\in\mathbb{N}}$ forms a descending chain of finite dimensional vector spaces over $k$. We know that$$\cap_{n\in\mathbb{N}}\operatorname{Hom}_{\mathcal{O}_{\bar{U}}}(\bar{\pi}^{*}\mathcal{G},\mathcal{I}^{n})\cong\operatorname{Hom}_{\mathcal{O}_{\bar{U}}}(\bar{\pi}^{*}\mathcal{G},\cap_{n\in\mathbb{N}}\mathcal{I}^{n}).$$It suffices to show the right-hand side is zero. Let $s\in\operatorname{Hom}_{\mathcal{O}_{\bar{U}}}(\bar{\pi}^{*}\mathcal{G},\cap_{n\in\mathbb{N}}\mathcal{I}^{n})$. Suppose, for a contradiction, that $s\neq 0$. Let $\mathcal{H}$ be the image of $s$ in $\cap_{n\in\mathbb{N}}\mathcal{I}^{n}$. Then $\mathcal{H}$ is a nonzero coherent sheaf on $\bar{U}$. By Krull's intersection theorem, the support of $\mathcal{H}$ is contained in $U$ and is thus affine over $k$. Since $\mathcal{X}$ is proper over $k$, so is $\operatorname{Supp}(\mathcal{H})$. It follows that $\operatorname{Supp}(\mathcal{H})$ is finite over $k$. Therefore it consists of finitely many closed points. Let $u$ be one of those closed points. Then $\mathcal{O}_{U,u}$ has a submodule $\mathcal{H}_{u}$. This contradicts the assumption that $\mathcal{O}_{U,u}$ has depth at least 1. Therefore $s=0$. It follows that$$\operatorname{Hom}_{\mathcal{O}_{\bar{U}}}(\bar{\pi}^{*}\mathcal{G},\cap_{n\in\mathbb{N}}\mathcal{I}^{n})=0$$as desired.
\end{proof}

\section{Functors between categories of coherent sheaves}\label{3}

In this section, we establish Theorem \ref{1.2}. Results in this section are essentially stacky variants of those in \cite[\href{https://stacks.math.columbia.edu/tag/0FZK}{Tag 0FZK}]{stacks-project}. We used several technical results in this section, which we prove in appendix \ref{A}. Let $F\colon\operatorname{QCoh}(\mathcal{X})\to\operatorname{QCoh}(\mathcal{Y})$ be an $R$-linear exact functor that commutes with arbitrary direct sums. We first show, with some reasonable assumptions on $\mathcal{X}$, that the corresponding quasi-coherent sheaf $\mathcal{K}$ admits a surjection from $\mathcal{F}\boxtimes F(\mathcal{O}_{\mathcal{X}})$, where $\mathcal{F}$ is a coherent sheaf on $\mathcal{X}$. We first remark that if $\mathcal{X}$ is an algebraic space that is separated and flat over $R$, then the proof of \cite[\href{https://stacks.math.columbia.edu/tag/0FZJ}{Tag 0FZJ}]{stacks-project} is still valid. In this case, it can be arranged that $\mathcal{F}=\mathcal{O}_{\mathcal{X}}$. Unfortunately, the same argument fails for stacks as the diagonal of a separated algebraic stack is only proper.

\begin{lemma}\label{3.1}
Let $R$ be a ring. Let $\mathcal{X}$ and $\mathcal{Y}$ be two algebraic stacks over $R$. Let $F\colon\operatorname{QCoh}(\mathcal{X})\to\operatorname{QCoh}(\mathcal{Y})$ be an $R$-linear exact functor that commutes with arbitrary direct sums. Let $\mathcal{K}$ be the object in $\operatorname{QCoh}(\mathcal{X}\times_{R}\mathcal{Y})$ corresponding to $F$. Suppose that $\mathcal{X}$ is tame and flat over $R$ with finite diagonal. Then there exists a surjective morphism$$\mathcal{E}\boxtimes F(\mathcal{G})\to\mathcal{K}$$for some coherent sheaf $\mathcal{E}$ and $\mathcal{G}$ on $\mathcal{X}$.
\end{lemma}

\begin{proof}
We first note that $\mathcal{X}$ is tame over $R$. Then the structure morphism $\mathcal{X}\to\operatorname{Spec}R$ is concentrated by \cite[Lemma\ 2.5(5)]{HR17}. It follows that the projection $p_{13}\colon\mathcal{X}\times_{R}\mathcal{X}\times_{R}\mathcal{Y}\to\mathcal{X}\times_{R}\mathcal{Y}$ is also concentrated by \cite[Lemma\ 2.5(1)]{HR17}. By Corollary \ref{2.4}, there exists coherent sheaves $\mathcal{E},\mathcal{G}$ on $\mathcal{X}$ with a surjective map $f\colon\mathcal{E}\boxtimes\mathcal{G}\to\mathcal{O}_{\Delta_{\mathcal{X}}}$, where $\mathcal{O}_{\Delta_{\mathcal{X}}}$ is the structure sheaf of the diagonal. Consider the functor$$(p_{13})_{*}(p_{12}^{*}(-)\otimes p_{23}^{*}\mathcal{K})\colon\operatorname{QCoh}(\mathcal{X}\times_{R}\mathcal{X}\times_{R}\mathcal{Y})\to\operatorname{QCoh}(\mathcal{X}\times_{R}\mathcal{Y}).$$We first show that this functor is exact. Since $\mathcal{X}$ is flat over $R$ and $\mathcal{K}$ is flat over $\mathcal{X}$, we only have to worry about $(p_{13})_{*}$. Applying \cite[\href{https://stacks.math.columbia.edu/tag/0FZI}{Tag 0FZI}]{stacks-project} after a flat base change gives us $R^{i}(p_{13})_{*}(p_{12}^{*}\mathcal{F}\otimes p_{23}^{*}\mathcal{K})=0$ for all $i>0$. This shows that our functor is exact.

Now we apply this functor to $f$, which gives us the following surjective morphism$$p_{13,*}(p_{12}^{*}(\mathcal{E}\boxtimes\mathcal{G})\otimes p_{23}^{*}\mathcal{K})\to p_{13,*}(p_{12}^{*}\mathcal{O}_{\Delta_{\mathcal{X}}}\otimes p_{23}^{*}\mathcal{K})$$in $\operatorname{QCoh}(\mathcal{X}\times_{R}\mathcal{Y})$. For the target of this morphism, we have
\begin{align*}
    p_{13,*}(p_{12}^{*}\mathcal{O}_{\Delta_{\mathcal{X}}}\otimes p_{23}^{*}\mathcal{K})&= p_{13,*}(p_{12}^{*}\Delta_{\mathcal{X},*}\mathcal{O}_{\mathcal{X}}\otimes p_{23}^{*}\mathcal{K})\\
    &\cong p_{13,*}((\Delta_{\mathcal{X}}\times_{R}id_{\mathcal{Y}})_{*}\mathcal{O}_{\mathcal{X}\times_{R}\mathcal{Y}}\otimes p_{23}^{*}\mathcal{K})\\
    &\cong p_{13,*}(\Delta_{\mathcal{X}}\times_{R}id_{\mathcal{Y}})_{*}((\Delta_{\mathcal{X}}\times_{R}id_{\mathcal{Y}})^{*}p_{23}^{*}\mathcal{K})\\
    &\cong p_{13,*}(\Delta_{\mathcal{X}}\times_{R}id_{\mathcal{Y}})_{*}\mathcal{K}\\
    &\cong\mathcal{K},
\end{align*}
where $\Delta_{\mathcal{X}}$ is the diagonal morphism of $\mathcal{X}$. For the source of this morphism, we have
\begin{align*}
    p_{13,*}(p_{12}^{*}(\mathcal{E}\boxtimes\mathcal{G})\otimes p_{23}^{*}\mathcal{K})&\cong p_{13,*}(p_{12}^{*}p_{1}^{*}\mathcal{E}\otimes p_{12}^{*}p_{2}^{*}\mathcal{G}\otimes p_{23}^{*}\mathcal{K})\\
    &\cong p_{13,*}(p_{13}^{*}p_{1}^{*}\mathcal{E}\otimes p_{23}^{*}(p_{2}^{*}\mathcal{G}\otimes\mathcal{K}))\\
    &\simeq Rp_{13,*}(p_{13}^{*}p_{1}^{*}\mathcal{E}\otimes p_{23}^{*}(p_{2}^{*}\mathcal{G}\otimes\mathcal{K}))\\
    &\simeq p_{1}^{*}\mathcal{E}\otimes Rp_{13,*}p_{23}^{*}(p_{2}^{*}\mathcal{G}\otimes\mathcal{K})\ \ (\text{\cite[Corollary\ 4.12]{HR17}})\\
    &\simeq p_{1}^{*}\mathcal{E}\otimes p_{13,*}p_{23}^{*}(p_{2}^{*}\mathcal{G}\otimes\mathcal{K})\\
    &\simeq p_{1}^{*}\mathcal{E}\otimes p_{3}^{*}p_{3,*}(p_{2}^{*}\mathcal{G}\otimes\mathcal{K})\\
    &\simeq p_{1}^{*}\mathcal{E}\otimes p_{3}^{*}F(\mathcal{G})\\
    &\simeq\mathcal{E}\boxtimes F(\mathcal{G}).
\end{align*}
This finishes the proof.
\end{proof}

See \cite[Definition\ 2.4]{HR17} for the precise definition of concentrated morphisms of algebraic stacks. Lemma \ref{3.1} can be regarded as a generalization of \cite[\href{https://stacks.math.columbia.edu/tag/0FZJ}{Tag 0FZJ}]{stacks-project} since any quasi-compact quasi-separated morphism of algebraic spaces is concentrated \cite[Lemma\ 2.5(3)]{HR17}. The following lemma states that any functor between categories of coherent sheaves on algebraic stacks can be extended to one between the corresponding categories of quasi-coherent sheaves.

\begin{lemma}[cf.\ {\cite[\href{https://stacks.math.columbia.edu/tag/0FZL}{Tag 0FZL}]{stacks-project}}]\label{3.2}
Let $\mathcal{X},\mathcal{Y}$ be Noetherian algebraic stacks. Let $F\colon\operatorname{Coh}(\mathcal{X})\to\operatorname{Coh}(\mathcal{Y})$ be an additive functor. There exists a unique extension $\bar{F}\colon\operatorname{QCoh}(\mathcal{X})\to\operatorname{QCoh}(\mathcal{Y})$ of $F$ that commutes with arbitrary direct sums. In particular, if $F$ is flat, so is $\bar{F}$. 
\end{lemma}

\begin{proof}
By assumption, every coherent sheaf on $\mathcal{X},\mathcal{Y}$ can be written as a sum of its coherent subsheaves. Then the argument of \cite[\href{https://stacks.math.columbia.edu/tag/0FZL}{Tag 0FZL}]{stacks-project} follows verbatim. 
\end{proof}

Let $\mathcal{X}$ be an algebraic stack. Let $x\colon \operatorname{Spec}l\to\mathcal{X}$ be a point of $\mathcal{X}$. If $\mathcal{X}$ is quasi-seperated, then there is a factorization$$\operatorname{Spec}l\xrightarrow{p}B_{l}G_{x}\xrightarrow{f}\mathcal{G}_{x}\xrightarrow{i}_{x}\mathcal{X},$$where $G_{x}\cong\operatorname{Aut}(x)$, the morphisms $p$ and $f$ are faithfully flat and $i_{x}$ is a quasi-affine monomorphism \cite[Theorem\ B.2]{quasifinite}. Note that the morphism $i_{x}\colon\mathcal{G}_{x}\to\mathcal{X}$ only depends on isomorphism classes of $x$. If $x$ is a closed point of $\mathcal{X}$, then $i_{x}$ is a closed immersion. We say $\mathcal{G}_{x}$ is the \textit{residual gerbe} of $\mathcal{X}$ at $x$. We set $\mathcal{O}_{x}=\overline{x}_{*}\mathcal{O}_{\operatorname{Spec}l}$. We would like to introduce the following technical lemmas.

\begin{lemma}\label{3.3}
Let $\mathcal{X}$ be an algebraic stack of finite type over $k$ with affine stabilizers. Let $x_{1}\cdots x_{n}$ be a finite sequence of closed points of $\mathcal{X}$. If $\mathcal{X}$ has quasi-finite diagonal, then there exists a single coherent sheaf $\mathcal{F}$ on $\mathcal{X}$ with a surjective morphism$$\mathcal{F}\to\bigoplus_{i=1}^{n}\mathcal{O}_{x_{i}}.$$
\end{lemma}

\begin{proof}
By \cite[Theorem\ 2.7]{EHKV01}, there exists a finite covering $\pi\colon W\to\mathcal{X}$ with $W$ a scheme. For any $i$, pick a closed point $u_{i}$ of $\pi^{-1}(x_{i})$. Then there exists a representative $\overline{x}_{i}\colon\operatorname{Spec}\kappa(u_{i})\to\mathcal{X}$ with a following factorization through the residual gerbe$$\operatorname{Spec}l_{i}\xrightarrow{p}\mathcal{G}_{x}\xhookrightarrow{j}\mathcal{X}$$where $p_{i}$ is flat and surjective and $j_{i}$ is a closed immersion. Then we have the following commutative diagram
\begin{equation*}
    \begin{tikzcd}
    & \pi^{-1}(\mathcal{G}_{x_{i}})\arrow[r,hookrightarrow,"j_{i}^{\prime}"]\arrow[d,"\pi_{i}"] & W\arrow[d,"\pi"]\\
    \operatorname{Spec}\kappa(u_{i})\arrow[ru,"\Tilde{p}_{i}"]\arrow[r,"p_{i}"] & \mathcal{G}_{x_{i}}\arrow[r,hookrightarrow,"j_{i}"] & \mathcal{X},
    \end{tikzcd}
\end{equation*}
where the right square is Cartesian. The composition $j_{i}^{\prime}\circ\Tilde{p}_{i}$ corresponds to a closed point $w_{i}$ of $W$. Since $i$ is arbitrary, we can produce a sequence of closed points $w_{1}\cdots w_{n}$ on $W$ as above. This gives us a surjective map$$\mathcal{O}_{W}\to\bigoplus_{i=1}^{n} j^\prime_{i,*}\Tilde{p}_{i,*}\kappa(u_{i})$$of coherent sheaves on $W$. Since $\pi$ is finite, $\pi_{*}$ is exact. Therefore we get a surjective map$$\pi_{*}\mathcal{O}_{W}\to\bigoplus_{i=1}^{n}\pi_{*}j^\prime_{i,*}\Tilde{p}_{i,*}\kappa(u_{i})$$of coherent sheaves on $\mathcal{X}$. Chasing the diagram gives$$\pi_{*}j^\prime_{i,*}\Tilde{p}_{i,*}\kappa(u_{i})\cong j_{i,*}p_{i,*}\kappa(u_{i})=\mathcal{O}_{x_{i}}.$$Setting $\mathcal{F}=\pi_{*}\mathcal{O}_{W}$ finishes the proof. 
\end{proof}

\begin{lemma}[cf.\ {\cite[\href{https://stacks.math.columbia.edu/tag/0FZM}{Tag 0FZM}]{stacks-project}}]\label{3.4}
Let $f\colon\mathcal{Z}\to\mathcal{Y}$ be a quasi-finite separated morphism of Noetherian algebraic stacks with finite relative inertia. Let $\mathcal{K}$ be a coherent sheaf on $\mathcal{Z}$ whose support is $\mathcal{Z}$ such that $f_{*}\mathcal{K}$ is coherent and $R^{i}f_{*}\mathcal{K}=0$ for all $i>0$. If $f$ is concentrated, then it is also proper.
\end{lemma}

\begin{proof}
Choose a smooth presentation $U\to\mathcal{Y}$. Then the pulled back morphism $f_{U}\colon\mathcal{Z}_{U}\to U$ is also quasi-finite separated with finite relative inertia. By Keel--Mori theorem, there exists an algebraic space $Z_{U}$ and a factorization $\mathcal{Z}_{U}\overset{\phi_{U}}{\to}Z_{U}\overset{\overline{f}_{U}}{\to}U$ such that $\phi_{U}$ is a proper and quasi-finite homeomorphism, and $\overline{f}_{U}$ is separated. Since $f_{U}$ is quasi-finite and $\phi_{U}$ is surjective, the morphism $\overline{f}_{U}$ is also quasi-finite. By Zariski's main theorem for algebraic spaces \cite[\href{https://stacks.math.columbia.edu/tag/082K}{Tag 082K}]{stacks-project}, we conclude that $Z_{U}$ is a scheme.

Let $\mathcal{K}^\prime$ be the coherent sheaf on $\mathcal{Z}_{U}$ obtained by pulling back $\mathcal{K}$ along the canonical projection. By assumption, the support of $\mathcal{K}^\prime$ is $\mathcal{Z}_{U}$. Since $f$ is concentrated, so is $f_{U}$, and $\phi_{U}$ is a good moduli space. It follows that $\phi_{U,*}$ is exact. Flat base change gives us$$R^{i}\overline{f}_{U,*}(\phi_{U,*}\mathcal{K}^\prime)=R^{i}f_{U,*}\mathcal{K}^\prime=0$$for all $i\geq 0$. Since $\phi_{U}$ is proper and surjective, $\phi_{U,*}\mathcal{K}^\prime$ is coherent, and the support of $\phi_{U,*}\mathcal{K}^\prime$ is $Z_{U}$. By \cite[\href{https://stacks.math.columbia.edu/tag/0FZM}{Tag 0FZM}]{stacks-project}, $\overline{f}_{U}$ is finite. Therefore $f_{U}$ is proper. Hence we may conclude that $f$ is proper.
\end{proof}

Now we are ready to prove Theorem \ref{1.2}.

\begin{proof}[Proof of Theorem \ref{1.2}]
Let $\mathcal{K}$ be a coherent sheaf as in Proposition \ref{1.2}\ref{1.3(2)}. The functor $F_{\mathcal{K}}$ is $k$-linear exact and sends $\operatorname{Coh}(\mathcal{X})$ to $\operatorname{Coh}(\mathcal{Y})$. So it suffices to find a quasi-inverse for this construction.

Let $F$ be a functor as in Theorem \ref{1.2}\ref{1.3(1)}. By Lemma \ref{A.3}, we may assume $F$ is a $k$-linear exact functor on the category of quasi-coherent sheaves that commute with arbitrary direct sums. By Proposition \ref{A.4}, there exists a quasi-coherent sheaf $\mathcal{K}$ on $\mathcal{X}\times_{k}\mathcal{Y}$ that is flat over $\mathcal{X}$ and  $R^{i}p_{2,*}(p_{1}^{*}(\mathcal{E})\otimes\mathcal{K})=0$ for all $i>0$ and for any coherent sheaf $\mathcal{E}$ on $\mathcal{X}$. By Lemma \ref{3.1}, there is a surjective map $\mathcal{E}\boxtimes F(\mathcal{G})\to\mathcal{K}$. It follows that $\mathcal{K}$ is coherent.

Now we show the support of $\mathcal{K}$ is proper and quasi-finite over $\mathcal{Y}$. Let $\mathcal{Z}$ be the support of $\mathcal{K}$. Let $y$ be a closed point of $\mathcal{Y}$. We claim that its fiber $\mathcal{Z}_{y}$ is discrete. Suppose, for a contradiction, that $\mathcal{Z}_{y}$ is not discrete. Then there exists an infinite sequence of pairwise distinct closed points $x_{1},x_{2},\cdots$ of $\mathcal{X}$ in the image of $\mathcal{Z}_{y}\to\mathcal{X}$. By Lemma \ref{3.3}, we have a surjective map$$\mathcal{F}\to\bigoplus_{i=1}^{N}\mathcal{O}_{x_{i}}$$for any positive integer $N$. Note that the sheaf $F(\mathcal{F})$ is coherent by assumption. Choose a presentation $p\colon\operatorname{Spec}(A)\to\mathcal{Y}$. Then we have a surjective map$$p^{*}F(\mathcal{F})\to\bigoplus_{i=1}^{N}p^{*}F(\mathcal{O}_{x_{i}})$$of coherent sheaves on $\operatorname{Spec}(A)$. This leads to a contradiction as we can choose $N$ to be larger than the number of generators of $p^{*}F(\mathcal{F})$ as a coherent $A$-module. This shows that $\mathcal{Z}_{y}$ is discrete, and thus $\mathcal{Z}$ is quasi-finite over $\mathcal{Y}$. Since $\mathcal{X}$ has finite diagonal, it is separated, and the morphism $\mathcal{X}\times\mathcal{Y}\to\mathcal{Y}$ is also separated. Therefore $\mathcal{Z}$ is quasi-finite and separated over $\mathcal{Y}$. By Lemma \ref{3.4}, $\mathcal{Z}\to\mathcal{Y}$ is also proper.
\end{proof}

Theorem \ref{1.2} can be regarded as a generalization of \cite[\href{https://stacks.math.columbia.edu/tag/0FZN}{Tag 0FZN}]{stacks-project}. Indeed, any finite morphism of algebraic stacks is proper and quasi-finite. However, the converse is not true without representability. Consider the classifying stack $BG$ where $G$ is a finite algebraic group over a field $k$. Then the structure morphism $BG\to\operatorname{Spec}k$ is proper and quasi-finite. But it is clearly not finite as $BG$ is not a scheme. We finish this section by verifying Theorem \ref{1.2} with the counterexample to Gabriel's theorem mentioned in the introduction.

\begin{example}\label{3.5}
Consider the classifying stack $B(\mathbb{Z}/3\mathbb{Z})$ over $\mathbb{C}$. Coherent sheaves on $B(\mathbb{Z}/3\mathbb{Z})$ are the same as finite dimensional complex representations of $\mathbb{Z}/3\mathbb{Z}$. It has three irreducible complex representations $V_{0}$, $V_{1}$, and $V_{2}$ corresponding to eigenvalues 1, $e^{\frac{2\pi i}{3}}$ and $e^{\frac{4\pi i}{3}}$ respectively. Note that $\mathbb{Z}/3\mathbb{Z}$ is linearly reductive, so the category $\operatorname{Coh}(B(\mathbb{Z}/3\mathbb{Z}))$ is semi-simple. Consider the functor $F\colon\operatorname{Coh}(B(\mathbb{Z}/3\mathbb{Z}))\to\operatorname{Coh}(B(\mathbb{Z}/3\mathbb{Z}))$ that fixes $V_{0}$ but exchanges $V_{1}$ and $V_{2}$. Clearly, $F$ is a well-defined exact functor. Consider the following sheaf$$K=(V_{0}^{\vee}\boxtimes V_{0})\oplus(V_{1}^{\vee}\boxtimes V_{2})\oplus(V_{2}^{\vee}\boxtimes V_{1}).$$Note that $K$ is coherent as it is a finite direct sum of coherent sheaves on $B(\mathbb{Z}/3\mathbb{Z})$. We claim that $F$ is given by the Fourier--Mukai transform with kernel $K$. Indeed, Schur's lemma gives us$$V_{i}\otimes V_{j}^{\vee}=
\begin{cases}
\mathbb{C} & i=j\\
0 & \text{otherwise}
\end{cases}$$for any $i,j\in\{0,1,2\}$. This tells us that $\Phi_{K}(V_{0})=V_{0}$, $\Phi_{K}(V_{1})=V_{2}$ and $\Phi_{K}(V_{2})=V_{1}$. It follows that $F\cong\Phi_{K}$.
\end{example}

In fact, the construction in Example \ref{3.5} works for any finite linearly reductive algebraic group $G$. If $G$ is infinite, we can still produce a Fourier--Mukai kernel for $F$ as above. However, it may not to be coherent. 

\begin{example}\label{3.6}
Consider the classifying stack $B\mathbb{G}_{m}$ over $\mathbb{C}$. As an algebraic group, $\mathbb{G}_{m}$ is also linearly reductive, and any irreducible representation of $\mathbb{G}_{m}$ is completely determined by a character $\mathbb{G}_{m}\to \mathbb{G}_{m}$. We write $V_{n}$ for the irreducible representation given by $t\to t^{n}$. Consider the functor $G\colon\operatorname{Coh}(B\mathbb{G}_{m})\to\operatorname{Coh}(B\mathbb{G}_{m})$ that fixes $V_{0}$ but exchanges $V_{-n}$ and $V_{n}$ for all $n>0$. In this case, $G$ is still a Fourier--Mukai transform. Consider the following sheaf$$K=\Bigg(\bigoplus_{i>0} (V_{i}^{\vee}\boxtimes V_{-i})\oplus(V_{-i}^{\vee}\boxtimes V_{i})\Bigg)\oplus(V_{0}^{\vee}\boxtimes V_{0}).$$This same argument in Example \ref{3.5} shows that $G$ is the Fourier--Mukai transform whose kernel is $K$. However, we notice that $K$ is no longer coherent as it is an infinite direct sum of coherent sheaves.
\end{example}

\section{Producing a kernel}\label{4}

In this section, we produce a candidate for the Fourier--Mukai kernel corresponding to a $k$-linear exact fully faithful functor $F\colon D^{b}_{coh}(\mathcal{X})\to D^{b}_{coh}(\mathcal{Y})$, and show that they agree on zero-dimensional sheaves on $\mathcal{X}$. Let $\mathcal{X},\mathcal{Y}$ be a smooth tame algebraic stack of finite type over $k$. It follows from \cite[Theorem\ B and C]{HR15} that both $\mathcal{X}$ and $\mathcal{Y}$ have finite cohomological dimensions. Assume further that $\mathcal{X}$ has finite diagonal. By Corollary \ref{2.4}, there exists a resolution of the diagonal of $\mathcal{X}$ as follows:
\begin{equation*}
    \begin{tikzcd}
    \cdots\arrow[r] & \mathcal{E}_{i}\boxtimes\mathcal{G}_{i}\arrow[r] & \cdots\arrow[r] & \mathcal{E}_{0}\boxtimes\mathcal{G}_{0}\arrow[r] & \mathcal{O}_{\Delta_{\mathcal{X}}}\arrow[r] & 0,
    \end{tikzcd}
\end{equation*}
where $\mathcal{E}_{i},\mathcal{G}_{i}$ are coherent sheaves on $\mathcal{X}$ for all $i$. Consider the following complex in $D^{b}_{coh}(\mathcal{X}\times_{k}\mathcal{X})$
\begin{equation}\label{eq4.1}
    \begin{tikzcd}
    \cdots\arrow[r] & \mathcal{E}_{i}\boxtimes\mathcal{G}_{i}\arrow[r] & \cdots\arrow[r] & \mathcal{E}_{0}\boxtimes\mathcal{G}_{0}.
    \end{tikzcd}
\end{equation}
For each $n$, we set $M_{n}$ to be the complex
\begin{equation*}
    \begin{tikzcd}
    \mathcal{E}_{n}\boxtimes\mathcal{G}_{n}\arrow[r] & \cdots\arrow[r] & \mathcal{E}_{0}\boxtimes\mathcal{G}_{0}
    \end{tikzcd}
\end{equation*}
obtained by taking the nth brutal truncation of the complex in \ref{eq4.1}. Then we have the following infinite Postnikov system
\begin{equation}\label{eq4.2}
    \begin{tikzcd}
   \cdots & M_{n}\arrow[dr]\arrow[l,dashed,"{[1]}"] & & \cdots\arrow[ll,dashed,"{[1]}"]\arrow[dr] & & M_{0}\arrow[ll,dashed,"{[1]}"]\arrow[dr,"\cong"]\\
    & \cdots\arrow[r] & \mathcal{E}_{n}\boxtimes\mathcal{G}_{n}\arrow[ur]\arrow[rr] & & \cdots\arrow[ur]\arrow[rr]& & \mathcal{E}_{0}\boxtimes\mathcal{G}_{0},
    \end{tikzcd}
\end{equation}
where the dashed arrow represents the shift in $D^{b}_{coh}(\mathcal{X}\times_{k}\mathcal{X})$. Set$$\mathcal{F}_{n}=\operatorname{Ker}(\mathcal{E}_{n}\boxtimes\mathcal{G}_{n}\to\mathcal{E}_{n-1}\boxtimes\mathcal{G}_{n-1})$$for all $n>0$. Then we have$$H^{i}(M_{n})=
\begin{cases}
\mathcal{O}_{\Delta_{\mathcal{X}}} & i=0\\
\mathcal{F}_{n} & i=-n\\
0 & i\neq 0,-n.
\end{cases}$$
Therefore we have $M_{n}\simeq\mathcal{O}_{\Delta_{\mathcal{X}}}\oplus\mathcal{F}_{n}[n]$ as complexes in $D^{b}_{coh}(\mathcal{X}\times_{k}\mathcal{X})$ for $n>\dim(\mathcal{X}\times_{k}\mathcal{X})$.

Now let $F\colon D^{b}_{coh}(\mathcal{X})\to D^{b}_{coh}(\mathcal{X})$ be a $k$-linear exact fully faithful functor. We wish to consider the complex
\begin{equation}\label{eq4.3}
    \begin{tikzcd}
    \cdots\arrow[r] & \mathcal{E}_{i}\boxtimes F(\mathcal{G}_{i})\arrow[r] & \cdots\arrow[r] & \mathcal{E}_{0}\boxtimes F(\mathcal{G}_{0}).
    \end{tikzcd}
\end{equation}
obtained by applying our functor to the second factor in the external tensor products in the complex in \ref{eq4.2}. But we need to check that it is a complex in $D^{b}_{coh}(\mathcal{X}\times_{k}\mathcal{Y})$. By the K\"unneth formula, we have
\begin{align*}
    \operatorname{Ext}_{\mathcal{O}_{\mathcal{X}\times_{k}\mathcal{Y}}}^{q}(\mathcal{E}_{i}\boxtimes F(\mathcal{G}_{i}),\mathcal{E}_{j}\boxtimes F(\mathcal{G}_{j}))&\simeq\bigoplus_{p\in\mathbb{Z}}\operatorname{Ext}_{\mathcal{O_{X}}}^{q+p}(\mathcal{E}_{i},\mathcal{E}_{j})\otimes_{k}\operatorname{Ext}_{\mathcal{O_{Y}}}^{-p}(F(\mathcal{G}_{i}),F(\mathcal{G}_{j}))\\
    &\simeq\bigoplus_{p\in\mathbb{Z}}\operatorname{Ext}_{\mathcal{O_{X}}}^{q+p}(\mathcal{E}_{i},\mathcal{E}_{j})\otimes_{k}\operatorname{Ext}_{\mathcal{O_{Y}}}^{-p}(\mathcal{G}_{i},\mathcal{G}_{j})\\
    &\simeq\operatorname{Ext}_{\mathcal{O_{X\times X}}}^{q}(\mathcal{E}_{i}\boxtimes \mathcal{G}_{i},\mathcal{E}_{j}\boxtimes \mathcal{G}_{j}).
\end{align*}
This shows that our complex in \ref{eq4.3} is indeed a complex in $D^{b}_{coh}(\mathcal{X}\times_{k}\mathcal{Y})$. Note that the second isomorphism holds because $F$ is fully faithful. Moreover, we have$$\operatorname{Ext}_{\mathcal{O}_{\mathcal{X}\times_{k}\mathcal{Y}}}^{q}(\mathcal{E}_{i}\boxtimes F(\mathcal{G}_{i}),\mathcal{E}_{j}\boxtimes F(\mathcal{G}_{j}))=0,\ q<0.$$By \cite[Lemma\ 2.1]{kawamata}, there exists an infinite Postnikov system
\begin{equation}\label{eq4.4}
    \begin{tikzcd}
   \ & K_{n}\arrow[dr]\arrow[l,dashed,"{[1]}"] & & \cdots\arrow[ll,dashed,"{[1]}"]\arrow[dr] & & K_{0}\arrow[ll,dashed,"{[1]}"]\arrow[dr,"\cong"]\\
    & \cdots\arrow[r] & \mathcal{E}_{n}\boxtimes F(\mathcal{G}_{n})\arrow[ur]\arrow[rr] & & \cdots\arrow[ur]\arrow[rr]& & \mathcal{E}_{0}\boxtimes F(\mathcal{G}_{0}).
    \end{tikzcd}
\end{equation}
We now work towards obtaining a direct sum decomposition for $K_{n}$ in $D^{b}_{coh}(\mathcal{X}\times_{k}\mathcal{Y})$ that is similar to the one for $M_{n}$ in $D^{b}_{coh}(\mathcal{X}\times_{k}\mathcal{X})$.

\begin{lemma}\label{4.1}
Let $\mathcal{X},\mathcal{Y}$ be smooth tame algebraic stacks of finite type over $k$. Let $F\colon D^{b}_{coh}(\mathcal{X})\to D^{b}_{coh}(\mathcal{Y})$ be a $k$-linear exact fully faithful functor. If $\mathcal{X}$ has finite diagonal, then there exists a positive integer $n$ and a coherent sheaf $\mathcal{K}_{n}$ such that$$\Phi_{K_{n}}^{\mathcal{X\to Y}}(\mathcal{F}))\simeq F(\mathcal{F})\oplus F(\mathcal{K}_{n})[n]$$for any zero dimensional sheaf $\mathcal{F}$ on $\mathcal{X}$.
\end{lemma}

\begin{proof}
Let $\mathcal{F}$ be a coherent sheaf on $\mathcal{X}$ with zero-dimensional support. Consider the following functor
\begin{align*}
    \Psi\colon D^{b}_{coh}(\mathcal{X}\times_{k}\mathcal{Y})&\longrightarrow D^{b}_{coh}(\mathcal{Y})\\
    N&\longmapsto\Phi_{N}^{\mathcal{X\to Y}}(\mathcal{F}).
\end{align*}
Note that this functor is exact. Applying the functor $\Psi$ to \ref{eq4.4} gives us a Postnikov system for the complex
\begin{equation}\label{eq4.5}
    \begin{tikzcd}
    \cdots\arrow[r] & \Phi_{\mathcal{E}_{n}\boxtimes F(\mathcal{G}_{n})}^{\mathcal{X\to Y}}(\mathcal{F})\arrow[r] & \cdots\arrow[r] & \Phi_{\mathcal{E}_{0}\boxtimes F(\mathcal{G}_{0})}^{\mathcal{X\to Y}}(\mathcal{F}).
    \end{tikzcd}
\end{equation}
Evaluating $\Phi_{\mathcal{E}_{n}\boxtimes F(\mathcal{G}_{n})}(\mathcal{F})$ gives us
\begin{align*}
    \Phi_{\mathcal{E}_{n}\boxtimes F(\mathcal{G}_{n})}^{\mathcal{X\to Y}}(\mathcal{F})&=Rp_{*}(Lq^{*}(\mathcal{F})\otimes_{\mathcal{O}_{\mathcal{X}\times_{k}\mathcal{Y}}}^{\mathbb{L}}(\mathcal{E}_{n}\boxtimes F(\mathcal{G}_{n}))\\
    &\simeq Rp_{*}(Lq^{*}(\mathcal{F}\otimes\mathcal{E}_{n})\otimes_{\mathcal{O}_{\mathcal{X}\times_{k}\mathcal{Y}}}^{\mathbb{L}}Lp^{*}F(\mathcal{G}_{n}))\\
    &\simeq Rp_{*}(Lq^{*}(\mathcal{F}\otimes\mathcal{E}_{n})\otimes_{\mathcal{O}_{\mathcal{X}\times_{k}\mathcal{Y}}}^{\mathbb{L}}Lp^{*}F(\mathcal{G}_{n}))\\
    &\simeq Rp_{*}Lq^{*}(\mathcal{F}\otimes\mathcal{E}_{n})\otimes_{\mathcal{O}_{\mathcal{X}\times_{k}\mathcal{Y}}}^{\mathbb{L}}F(\mathcal{G}_{n}))\\
    &\simeq\Gamma(\mathcal{X}, \mathcal{F}\otimes\mathcal{E}_{n})\otimes_{k}F(\mathcal{G}_{n}).
\end{align*}
The second last quasi-isomorphism follows from the projection formula for concentrated morphism \cite[Corollary\ 4.12]{HR17}. The last quasi-isomorphism is due to the fact that $\mathcal{F}\otimes\mathcal{E}_{n}$ has vanishing higher cohomology sheaves since it has zero-dimensional support.

On the other hand, consider another functor
\begin{align*}
    \Psi^\prime\colon D^{b}_{coh}(\mathcal{X\times X})&\longrightarrow D^{b}_{coh}(\mathcal{Y})\\
    N&\longmapsto F(\Phi_{N}^{\mathcal{X}\to \mathcal{X}}(\mathcal{F})).
\end{align*}
Again, this functor is exact by assumption. Applying the functor $\Psi^\prime$ to \ref{eq4.2} gives us a Postnikov system for the complex
\begin{equation}\label{eq4.6}
    \begin{tikzcd}
    \cdots\arrow[r] & F(\Phi_{\mathcal{E}_{n}\boxtimes\mathcal{G}_{n}}^{\mathcal{X}\to \mathcal{X}}(\mathcal{F}))\arrow[r] & \cdots\arrow[r] & F(\Phi_{\mathcal{E}_{0}\boxtimes\mathcal{G}_{0}}^{\mathcal{X}\to \mathcal{X}}(\mathcal{F})).
    \end{tikzcd}
\end{equation}
Evaluating $F(\Phi_{\mathcal{E}_{n}\boxtimes\mathcal{G}_{n}}(\mathcal{F}))$ gives us
\begin{align*}
    F(\Phi_{\mathcal{E}_{n}\boxtimes\mathcal{G}_{n}}^{\mathcal{X}\to \mathcal{X}}(\mathcal{F}))&\simeq F(\Gamma(\mathcal{X},\mathcal{F}\otimes\mathcal{E}_{n})\otimes_{k}\mathcal{G}_{n})\\
    &\simeq\Gamma(\mathcal{X},\mathcal{F}\otimes\mathcal{E}_{n})\otimes_{k}F(\mathcal{G}_{n}).
\end{align*}
Thus we have obtained two Postnikov systems for the same complex in $D^{b}_{coh}(\mathcal{Y})$. Moreover, we have$$\operatorname{Ext}_{\mathcal{O_{Y}}}^{q}(F(\Gamma(\mathcal{X},\mathcal{F}\otimes\mathcal{E}_{i})\otimes_{k}\mathcal{G}_{i}),F(\Gamma(\mathcal{X},\mathcal{F}\otimes\mathcal{E}_{j})\otimes_{k}\mathcal{G}_{j}))=0,\ q<0.$$By \cite[Lemma\ 2.1]{kawamata}, we conclude that$$\Phi_{K_{n}}^{\mathcal{X\to Y}}(\mathcal{F})\simeq F(\Phi_{M_{n}}^{\mathcal{X}\to \mathcal{X}}(\mathcal{F}))$$for all $n>0$. If $n>\dim(\mathcal{X}\times_{k}\mathcal{X})$, we have $M_{n}\simeq\mathcal{O}_{\Delta_{\mathcal{X}}}\oplus\mathcal{F}_{n}[n]$. It follows that$$\Phi_{K_{n}}^{\mathcal{X\to Y}}(\mathcal{F})\simeq F(\mathcal{F})\oplus F(\Phi_{\mathcal{F}_{n}}^{\mathcal{X}\to \mathcal{X}}(\mathcal{F}))[n].$$This finishes the proof.
\end{proof}

Now we produce the premised Fourier--Mukai kernel. We first relate the cohomology of $\Phi_{K_{n}}^{\mathcal{X}\to\mathcal{Y}}$ to that of $K_{n}$.

\begin{lemma}\label{4.2}
Let $\mathcal{X},\mathcal{Y}$ be smooth tame algebraic stacks of finite type over $k$. Let $K$ be a bounded complex in $D_{coh}^{b}(\mathcal{X})$. Assume that $\mathcal{X}$ has finite diagonal. For a fixed integer $i$, if $H^{i}(\Phi_{K}^{\mathcal{X}\to\mathcal{Y}}(\mathcal{F}))=0$ for all coherent sheaves $\mathcal{F}$ on $\mathcal{X}$ with zero dimensional support, then $H^{i}(K)=0$.
\end{lemma}

\begin{proof}
The problem is local on $\mathcal{Y}$, so we may assume $\mathcal{Y}$ is an affine scheme. By \cite[Theorem\ 7.2]{quasifinite}, we may choose a quasi-finite separated flat presentation $\pi_{1}\colon U\to\mathcal{X}$ with $U$ an affine scheme. Consider the following Cartesian diagram
\begin{equation*}
    \begin{tikzcd}
    U\arrow[d,"\pi_{1}",swap] & U\times\mathcal{Y}\arrow[l,"\tilde{p}_{1}",swap]\arrow[d,"\pi",swap]\arrow[rd,"\tilde{p}_{2}"] & \\
    \mathcal{X} & \mathcal{X}\times\mathcal{Y}\arrow[l,"p_{1}"]\arrow[r,"p_{2}",swap] & \mathcal{Y}.
    \end{tikzcd}
\end{equation*}
Let $\mathcal{G}$ be a coherent sheaf on $U$ with zero dimensional support. Since $\pi_{1}$ is quasi-finite, the pushforward sheaf $\pi_{1,*}\mathcal{G}$ is a coherent sheaf on $\mathcal{X}$ with zero-dimensional support. By assumption, we have $H^{i}(\Phi_{K}^{\mathcal{X}\to\mathcal{Y}}(\pi_{1,*}\mathcal{G}))=0$. Flat base change yields$$\Phi_{K}^{\mathcal{X}\to\mathcal{Y}}(\pi_{1,*}\mathcal{G})\simeq R\tilde{p}_{2,*}(\pi^{*}\pi_{*}\tilde{p}_{1}^{*}\mathcal{G}\otimes^{\mathbb{L}}L\pi^{*}K).$$Since $\mathcal{X}$ has finite diagonal, $\pi$ is also affine. So the counit of the adjunction$$\pi^{*}\pi_{*}\tilde{p}_{1}^{*}\mathcal{G}\to\tilde{p}_{1}^{*}\mathcal{G}$$is surjective. Then we have a morphism of complexes$$R\tilde{p}_{2,*}(\pi^{*}\pi_{*}\tilde{p}_{1}^{*}\mathcal{G}\otimes^{\mathbb{L}}L\pi^{*}K)\to\Phi_{L\pi^{*}K}^{U\to\mathcal{Y}}(\pi_{1,*}\mathcal{G})$$whose induces map on the cohomology sheaves$$H^{i}(\Phi_{K}^{\mathcal{X}\to\mathcal{Y}}(\pi_{1,*}\mathcal{G}))\to H^{i}(\Phi_{L\pi^{*}K}^{U\to\mathcal{Y}}(\pi_{1,*}\mathcal{G}))$$is surjective for all $i>0$. It follows that $H^{i}(\Phi_{\pi^{*}K}^{U\to\mathcal{Y}}(\pi_{1,*}\mathcal{G}))=0$. This is true for all $\mathcal{G}$ on $U$ with zero dimensional support. It follows from \cite[\href{https://stacks.math.columbia.edu/tag/0FZ9}{Tag 0FZ9}]{stacks-project} that $H^{i}(L\pi^{*}K)=0$. Since $\pi$ is faithfully flat, we conclude that $H^{i}(K)=0$ as desired.
\end{proof}

Recall that a $k$-linear exact functor $F\colon D^{b}_{coh}(\mathcal{X})\to D^{b}_{coh}(\mathcal{Y})$ is {\it bounded} if there exists a positive integer $m$ such that for all coherent sheaves $\mathcal{F}$ on $\mathcal{X}$, the coherent sheaf $F(\mathcal{F})$ has nonzero cohomology only in $[-m,m]$. If $D_{coh}^{b}(\mathcal{X})$ is $\operatorname{Ext}$-finite and has a strong generator, then any exact functor from $D_{coh}^{b}(\mathcal{X})$ is bounded by \cite[Theorem\ 1.3]{BvdB}. These conditions are satisfied when $\mathcal{X}$ is proper and tame over $k$. Indeed, properness ensures that $D_{coh}^{b}(\mathcal{X})$ is Ext-finite and strong generation follows from \cite[Theorem\ B.1]{HP24}.

\begin{lemma}\label{4.3}
Let $\mathcal{X},\mathcal{Y}$ be smooth tame algebraic stacks of finite type over $k$. Let $F\colon D^{b}_{coh}(\mathcal{X})\to D^{b}_{coh}(\mathcal{Y})$ be a $k$-linear exact fully faithful functor. If $\mathcal{X}$ is proper, then there exists a complex $K$ such that$$\Phi_{K}^{\mathcal{X\to Y}}(\mathcal{F})\simeq F(\mathcal{F})$$for any coherent sheaf $\mathcal{F}$ on $\mathcal{X}$ with zero dimensional support.
\end{lemma}

\begin{proof}
By assumption the diagonal of $\mathcal{X}$ is finite. Since $F$ is bounded, the complex $\Phi_{K_{n}}^{\mathcal{X\to Y}}$ has vanishing cohomology sheaves outside $[-m,m]\cup[-m-n,m-n]$. By Lemma \ref{4.2}, we know that $K_{n}$ has vanishing cohomology outside $[-m,m]\cup[-m-n,m-n]$. Choose $n>\dim(\mathcal{X}\times\mathcal{Y})+2m$. We have a decomposition$$K_{n}\simeq K\oplus C_{n}$$where $K$ is a bounded complex in $D_{coh}^{b}(\mathcal{X}\times\mathcal{Y})$ concentrated in $[-m,m]$ and $C_{n}$ is a bounded complex in $D_{coh}^{b}(\mathcal{X}\times\mathcal{Y})$ concentrated in $[-m-n,m-n]$. Therefore we have$$\Phi_{K_{n}}^{\mathcal{X\to Y}}(\mathcal{F})\simeq\Phi_{K}^{\mathcal{X\to Y}}(\mathcal{F})\oplus\Phi_{C_{n}}^{\mathcal{X\to Y}}(\mathcal{F}),$$where $\Phi_{K}^{\mathcal{X\to Y}}(\mathcal{F})$ is a bounded complex in $D_{coh}^{b}(\mathcal{Y})$ concentrated in $[-m,m]$ and $\Phi_{C_{n}}^{\mathcal{X\to Y}}(\mathcal{F})$ is a bounded complex in $D_{coh}^{b}(\mathcal{Y})$ concentrated in $[-m-n,m-n]$. By Lemma \ref{4.1}, we have$$\Phi_{K_{n}}^{\mathcal{X\to Y}}(\mathcal{F})\simeq F(\mathcal{F})\oplus F(\mathcal{K}_{n})[n].$$Comparing the cohomology groups gives us $\Phi_{K}^{\mathcal{X\to Y}}(\mathcal{F})\simeq F(\mathcal{F})$ and $\Phi_{C_{n}}^{\mathcal{X\to Y}}(\mathcal{F})\simeq F(\Phi_{\mathcal{F}_{n}}^{\mathcal{X}\to \mathcal{X}}(\mathcal{F}))[n]$ as desired.
\end{proof}

\section{Generalized points}\label{5}

In this section, we establish some auxiliary results about exact functors that preserve coherent sheaves with zero-dimensional support. We first recall the notion of generalized closed points, which is first introduced Lim and Polishchuk in \cite{LP21} over an algebraically closed field and then generalized Hall and Priver in \cite{HP24}.

\begin{definition}\label{5.a}
    Let $\mathcal{X}$ be a quasi-separated algebraic stack. A \textit{generalized point} of $\mathcal{X}$ consists of a pair $(x,\xi)$ where $x\colon\operatorname{Spec}l\to\mathcal{X}$ is a point of $\mathcal{X}$ and $\xi$ is a simple object in $\operatorname{QCoh}(\mathcal{G}_{x})$. We say a generalized point $(x,\xi)$ of $\mathcal{X}$ is close if $x$ is a closed point of $\mathcal{X}$. Let $\mathcal{O}_{x,\xi}=i_{x,*}\xi$. We call $\mathcal{O}_{x,\xi}$ the \textit{generalized skyscraper sheaf} supported at $(x,\xi)$.
\end{definition}

See \cite[Section\ 7]{HP24} for the details on the definition of generalized points. All the generalized points considered in this article are closed unless otherwise stated. It is well-known that the collections of all the skyscraper sheaves on a smooth algebraic variety $X$ over a field form a spanning class for $D_{coh}^{b}(X)$. The following result provides us with an analogous statement for generalized skyscraper sheaves on certain algebraic stacks. We say an algebraic stack $\mathcal{X}$ satisfies the Thomason condition if $D_{qc}(\mathcal{X})$ is compactly generated and for any open subset $U\subseteq\mathcal{X}$, there is a compact perfect complex supported on $\mathcal{X}\setminus U$. 

\begin{proposition}[{\cite[Corollary\ 7.17]{HP24}}]\label{5.1}
Let $\mathcal{X}$ be a Noetherian algebraic stack. If $\mathcal{X}$ satisfies the Thomason condition, then the collection of all the generalized skyscraper sheaves on $\mathcal{X}$ forms a spanning class for $D_{coh}^{b}(\mathcal{X})$.
\end{proposition}

See \cite[Example\ 7.16]{HP24} for examples of algebraic stacks satisfying the Thomason condition. In particular, any regular tame algebraic stack over a field $k$ with quasi-affine diagonal and finite Krull dimension satisfies the Thomason condition \cite[Theorem\ 2.1]{Hal22}.

In \cite[Section\ 3.2]{olander_2022}, Olander established some useful lemmas on fully faithful functors agreeing on points of algebraic spaces. For the rest of this section, we establish these lemmas for faithful functors agreeing on generalized points of tame algebraic stacks. The arguments for these results are completely the same. We decided to include them for completeness. We first produce both a left and a right adjoint for a fully faithful functor $F\colon D_{coh}^{b}(\mathcal{X})\to D_{coh}^{b}(\mathcal{Y})$.

\begin{lemma}\label{5.2}
Let $\mathcal{X},\mathcal{Y}$ be smooth proper, and tame algebraic stacks over $k$. Let $F\colon D_{coh}^{b}(\mathcal{X})\to D_{coh}^{b}(\mathcal{Y})$ be a $k$-linear exact functor. Then $F$ has both a left and right adjoint.
\end{lemma}

\begin{proof}
Since $\mathcal{Y}$ is proper, $D_{coh}^{b}(\mathcal{Y})$ is $\operatorname{Ext}$-finite. The statement follows from \cite[Corollary\ B.3]{HP24}.
\end{proof}

\begin{lemma}\label{5.3}
Let $\mathcal{X},\mathcal{Y}$ be smooth proper, and tame algebraic stacks over $k$. Let $F\colon D_{coh}^{b}(\mathcal{X})\to D_{coh}^{b}(\mathcal{Y})$ be a $k$-linear exact functor. If the natural map$$F^{*}\colon\operatorname{Ext}_{\mathcal{O}_{\mathcal{X}}}^{*}(\mathcal{O}_{x,\xi},\mathcal{O}_{y,\nu})\to\operatorname{Ext}_{\mathcal{O}_{\mathcal{Y}}}^{*}(F(\mathcal{O}_{x,\xi}),F(\mathcal{O}_{y,\nu}))$$is an isomorphism for all generalized points $(x,\xi)$ and $(y,\nu)$ of $\mathcal{X}$, then $F$ is fully faithful.
\end{lemma}

\begin{proof}
By Lemma \ref{5.2}, $F$ has both a left adjoint $L$ and a right adjoint $R$. Therefore for any pair of generalized points $(x,\xi)$ and $(y,\nu)$, we have$$\operatorname{Ext}_{\mathcal{O}_{\mathcal{X}}}^{*}(\mathcal{O}_{x,\xi},\mathcal{O}_{y,\nu})\cong\operatorname{Ext}_{\mathcal{O}_{\mathcal{Y}}}^{*}(F(\mathcal{O}_{x,\xi}),F(\mathcal{O}_{y,\nu}))\cong\operatorname{Ext}_{\mathcal{O}_{\mathcal{Y}}}^{*}(\mathcal{O}_{x,\xi},RF(\mathcal{O}_{y,\nu})).$$Since the generalied skyscraper sheaves form a spanning class of $D_{coh}^{b}(\mathcal{X})$, we conclude that the unit of the adjunction $id_{D_{coh}^{b}(\mathcal{X})}\to RF$ is an isomorphism on the generalized skyscraper sheaves. Let $K$ be an object in $D_{coh}^{b}(\mathcal{X})$ and $(x,\xi)$ a generalized point on $\mathcal{X}$. Then we have
\begin{align*}
    \operatorname{Ext}_{\mathcal{O}_{\mathcal{X}}}^{*}(LF(K),\mathcal{O}_{x,\xi})&\cong\operatorname{Ext}_{\mathcal{O}_{\mathcal{X}}}^{*}(F(K),F(\mathcal{O}_{x,\xi}))\\
    &\cong\operatorname{Ext}_{\mathcal{O}_{\mathcal{X}}}^{*}(K,RF(\mathcal{O}_{x,\xi}))\\
    &\cong\operatorname{Ext}_{\mathcal{O}_{\mathcal{X}}}^{*}(K,\mathcal{O}_{x,\xi}).
\end{align*}
Since the sheaves $\mathcal{O}_{x,\xi}$ form a spanning class of $D_{coh}^{b}(\mathcal{X})$, we conclude that the counit of the adjunction $LF\to id_{D_{coh}^{b}(\mathcal{X})}$ is an isomorphism. Hence, $F$ is fully faithful.
\end{proof}

\begin{lemma}\label{5.4}
Let $\mathcal{X},\mathcal{Y}$ be smooth proper and tame algebraic stacks over $k$. Let $F,G\colon D_{coh}^{b}(\mathcal{X})\to D_{coh}^{b}(\mathcal{Y})$ be $k$-linear exact functors. If $F$ is fully faithful and $F(\mathcal{O}_{x,\xi})\cong G(\mathcal{O}_{x,\xi})$ for every generalized point $(x,\xi)$ of $\mathcal{X}$, then the essential image of $G$ is contained in the essential image of $F$.
\end{lemma}

\begin{proof}
We know that $F$ and $G$ are both bounded, and $G$ has a right adjoint $R$ by Lemma \ref{5.2}. Let $\mathcal{A}$ be the essential image of $F$ in $D_{coh}^{b}(\mathcal{Y})$. Then $\mathcal{A}$ is an admissible subcategory of $D_{coh}^{b}(\mathcal{Y})$. Let $(x,\xi)$ be a generalized point of $\mathcal{X}$. Let $M$ be a complex in $\mathcal{A}^\perp$. Then we have$$\operatorname{Ext}_{\mathcal{O}_{\mathcal{Y}}}^{*}(F(\mathcal{O}_{x,\xi}),M)\cong\operatorname{Ext}_{\mathcal{O}_{\mathcal{Y}}}^{*}(G(\mathcal{O}_{x,\xi}),M)\cong\operatorname{Ext}_{\mathcal{O}_{\mathcal{X}}}^{*}(\mathcal{O}_{x,\xi},R(M))=0.$$Since the generalized skyscraper sheaves form a spanning class of $D_{coh}^{b}(\mathcal{X})$, we conclude that $R(M)=0$. Hence, the essential image of $G$ must be contained in $\mathcal{A}$.
\end{proof}

\begin{lemma}\label{5.5}
Let $\mathcal{X}$ be a smooth proper and tame algebraic stack over $k$. Let $F\colon D_{coh}^{b}(\mathcal{X})\to D_{coh}^{b}(\mathcal{X})$ be a $k$-linear exact functor. If $F(\mathcal{F})\cong\mathcal{F}$ for all coherent sheaves $\mathcal{F}$ on $\mathcal{X}$ with zero dimensional support, then $F$ is an equivalence.
\end{lemma}

\begin{proof}
We prove that $F$ fully faithful by showing the natural map$$F^{i}\colon\operatorname{Ext}_{\mathcal{O}_{\mathcal{X}}}^{i}(\mathcal{O}_{x,\xi},\mathcal{O}_{y,\nu})\to\operatorname{Ext}_{\mathcal{O}_{\mathcal{X}}}^{i}(F(\mathcal{O}_{x,\xi}),F(\mathcal{O}_{y,\nu}))$$is an isomorphism for all $i$ and all generalized points $(x,\xi)$ and $(y,\nu)$ of $\mathcal{X}$. If $(x,\xi)$ and $(y,\nu)$ are distinct, then both sides vanish. If $(x,\xi)=(y,\nu)$, we need to show that $F^{i}$ is an isomorphism for each $i$. 

We first observe that the source and target of $F^{i}$ are isomorphic to each other by our assumption. Therefore, it suffices to show that $F^{i}$ is surjective, or equivalently injective, for all $i$. For $i\leq 0$, both the source and target vanish, so there is nothing to prove. For $i=0$, we know that $\xi$ is simple. It follows from Schur's lemma that every element in $\operatorname{Hom}(\mathcal{O}_{x,\xi},\mathcal{O}_{x,\xi})$ is an isomorphism. Since $F$ is fully faithful, it preserves and reflects isomorphisms. It follows that $F^{0}$ is an isomorphism. For $i>1$, there exists a surjective map$$\operatorname{Ext}_{\mathcal{O}_{\mathcal{X}}}^{1}(\mathcal{O}_{x,\xi},\mathcal{O}_{x,\xi})\otimes\cdots\otimes\operatorname{Ext}_{\mathcal{O}_{\mathcal{X}}}^{1}(\mathcal{O}_{x,\xi},\mathcal{O}_{x,\xi})\to\operatorname{Ext}_{\mathcal{O}_{\mathcal{X}}}^{i}(\mathcal{O}_{x,\xi},\mathcal{O}_{x,\xi})$$given by composition of extension classes. This tells us that the surjectivity of $F^{1}$ would imply the surjectivity of $F^{i}$ for all $i>1$. 

It remains to show $F^{1}$ is surjective or equivalently injective. Observe that any nonzero element $\alpha$ in $\operatorname{Ext}_{\mathcal{O}_{\mathcal{X}}}^{1}(F(\mathcal{O}_{x,\xi}),F(\mathcal{O}_{x,\xi}))$ corresponds to a non-split extension$$0\rightarrow\mathcal{O}_{x,\xi}\rightarrow\mathcal{F}\rightarrow\mathcal{O}_{x,\xi}\rightarrow 0.$$Note that the support of $\mathcal{F}$ is also zero-dimensional. We see that $F(\alpha)$ corresponds to the extension$$0\rightarrow F(\mathcal{O}_{x,\xi})\rightarrow F(\mathcal{F})\rightarrow F(\mathcal{O}_{x,\xi})\rightarrow 0.$$By assumption, we have $F(\mathcal{O}_{x,\xi})\cong\mathcal{O}_{x,\xi}$, and $F(\mathcal{F})\cong\mathcal{F}$. This tells us that the extension corresponding to $F(\alpha)$ does not split either. So $F(\alpha)$ is nonzero. Therefore $F^{1}$ must be injective. This shows that $F$ is fully faithful by Lemma \ref{5.3}. Applying Lemma \ref{5.4} to $F$ and $id_{D_{coh}^{b}(\mathcal{X})}$, we know that $F$ is also essentially surjective. This finishes the proof.
\end{proof}

\begin{lemma}\label{5.6}
Let $\mathcal{X}$ be a smooth proper and tame algebraic stack over $k$. Let $F\colon D_{coh}^{b}(\mathcal{X})\to D_{coh}^{b}(\mathcal{X})$ be a $k$-linear exact functor. If $F(\mathcal{F})\cong\mathcal{F}$ for all coherent sheaves $\mathcal{F}$ on $\mathcal{X}$ with zero-dimensional support, then $F$ restricts to an autoequivalence on $\operatorname{Coh}(\mathcal{X})$.
\end{lemma}

\begin{proof}
By Lemma \ref{5.5}, it suffices to show that $F$ preserves coherent sheaves on $\mathcal{X}$. Let $\mathcal{F}$ be a coherent sheaf on $\mathcal{X}$. Since $F$ is fully faithful, we have$$\operatorname{Ext}_{\mathcal{O}_{\mathcal{X}}}^{i}(\mathcal{F},\mathcal{O}_{x,\xi})\cong\operatorname{Ext}_{\mathcal{O}_{\mathcal{X}}}^{i}(F(\mathcal{F}),F(\mathcal{O}_{x,\xi}))\cong\operatorname{Ext}_{\mathcal{O}_{\mathcal{X}}}^{i}(F(\mathcal{F}),\mathcal{O}_{x,\xi})$$for all $i$ and for all generalized points $(x,\xi)$. Therefore $F(\mathcal{F})$ has non-zero cohomology only in degree at most 0. Then there is a distinguished triangle$$\tau^{<0}F(\mathcal{F})\longrightarrow F(\mathcal{F})\longrightarrow H^{0}(F(\mathcal{F}))\longrightarrow\tau^{<0}F(\mathcal{F})[1].$$Since $F$ is an equivalence, we have a distinguished triangle$$F^{-1}\tau^{<0}F(\mathcal{F})\longrightarrow \mathcal{F}\longrightarrow F^{-1}H^{0}(F(\mathcal{F}))\longrightarrow F^{-1}\tau^{<0}F(\mathcal{F})[1].$$Note that $\tau^{<0}F(\mathcal{F})$ has nonzero cohomology only in degrees less than 0. We claim that so does $F^{-1}\tau^{<0}F(\mathcal{F})$. To see this, let $j$ be the largest integer such that $H^{j}(F^{-1}\tau^{<0}F(\mathcal{F}))\neq 0$. Then there exists a generalized point $(x,\xi)$ of $\mathcal{X}$ such that $\operatorname{Ext}_{\mathcal{O}_{\mathcal{X}}}^{-j}(F^{-1}\tau^{<0}F(\mathcal{F}),\mathcal{O}_{x,\xi})\neq 0$. By assumption $F$ is fully faithful and $F(\mathcal{O}_{x,\xi})\cong\mathcal{O}_{x,\xi}$. It follows that $\operatorname{Ext}_{\mathcal{O}_{\mathcal{X}}}^{-j}(\tau^{<0}F(\mathcal{F}),\mathcal{O}_{x,\xi})\neq 0$. This implies that $j<0$ as claimed. Therefore there are no nonzero maps from $F^{-1}\tau^{<0}F(\mathcal{F})$ to $\mathcal{F}$. Since $F$ is an equivalence, there are no nonzero maps from $\tau^{<0}F(\mathcal{F})$ to $F(\mathcal{F})$. This tells us that $F(\mathcal{F})$ is concentrated in degree 0.
\end{proof}

We conclude this section with an immediate consequence of Theorem \ref{1.2} and Lemma \ref{5.5}.

\begin{corollary}\label{5.7}
Let $\mathcal{X}$ be a smooth proper and tame algebraic stack over $k$. Let $F\colon D_{coh}^{b}(\mathcal{X})\to D_{coh}^{b}(\mathcal{X})$ be a $k$-linear exact functor. If $F(\mathcal{F})\cong\mathcal{F}$ for all coherent sheaves $\mathcal{F}$ on $\mathcal{X}$ with zero-dimensional support, then $F$ is naturally isomorphic to a Fourier--Mukai transform when restricted to $\operatorname{Coh}(\mathcal{X})$.
\end{corollary}

\section{Proof of Theorem \ref{1.1}}\label{6}

We first recall the notion of almost ample sets in a triangulated category due to Canonaco and Stellari.

\begin{definition}[{\cite[Definition\ 2.9]{CS14}}]\label{6.1}
Let $\mathcal{A}$ be an abelian category. We say a set of objects $\{\mathcal{F}_{i}\}_{i\in I}$ in $\mathcal{A}$ is \textit{almost ample} if, for any object $A$ in $\mathcal{A}$, there exists $i\in I$ such that
\begin{enumerate}[topsep=0pt,noitemsep,label=\normalfont(\arabic*)]
    \item there is a surjective morphism $\mathcal{F}_{i}^{\oplus k}\to A$ for some $k>0$, and
    \item $\operatorname{Hom}(A, \mathcal{F}_{i})=0$.
\end{enumerate}
\end{definition}

Let $\mathcal{A}$ be an Abelian category with almost ample sets. Let $F, G$ be two exact functors from $D^{b}(\mathcal{A})$ to another triangulated category. A nice feature of almost ample sets is that if $F$ is fully faithful and $F, G$ are naturally isomorphic when restricted to $\mathcal{A}$, then we $F, G$ must be naturally isomorphic. This is a special case of \cite[Proposition\ 3.3]{CS14}.

Let $\mathcal{X}$ be a quasi-compact algebraic stack of finite type over a field $k$ with finite diagonal. In section 2, we constructed a set $\{\mathcal{F}_{n}\}_{n\in\mathbb{N}}$ of objects in $\operatorname{Coh}(\mathcal{X})$. Then the set $\{\mathcal{F}_{n}\}_{n\in\mathbb{N}}$ is almost ample if $\mathcal{X}$ is reduced and proper with strictly positive dimension. This is an immediate consequence of Lemma \ref{2.2} and \ref{2.5}. Combining this observation with Lemma \ref{5.6}, we obtain the following result.

\begin{lemma}\label{6.2}
Let $\mathcal{X}$ be a smooth proper, and tame algebraic stack of finite type over $k$. Let $F\colon D_{coh}^{b}(\mathcal{X})\to D_{coh}^{b}(\mathcal{X})$ be a $k$-linear exact functor. If $F(\mathcal{F})\cong\mathcal{F}$ for all coherent sheaves $\mathcal{F}$ on $\mathcal{X}$ with zero dimensional support, then $F$ is a Fourier--Mukai autoequivalence.
\end{lemma}

\begin{proof}
It follows from Lemma \ref{5.6} that $F$ is an autoequivalence. It remains to show it is a Fourier--Mukai transform. Let $\{\mathcal{X}_{i}\}_{1\leq i\leq n}$ be the set of connected components of $\mathcal{X}$. There exists an orthogonal decomposition$$D_{coh}^{b}(\mathcal{X})\cong\bigoplus_{i=1}^{n} D_{coh}^{b}(\mathcal{X}_{i}).$$The composition$$D_{coh}^{b}(\mathcal{X}_{i})\hookrightarrow D_{coh}^{b}(\mathcal{X})\xrightarrow[]{F}D_{coh}^{b}(\mathcal{X})$$defines an autoequivalence on $D_{coh}^{b}(\mathcal{X}_{i})$ for all $i$. Therefore we may assume that $\mathcal{X}$ is connected. 

We first note that Corollary \ref{5.7} says that $F$ is naturally isomorphic to a Fourier--Mukai transform when restricted to $\operatorname{Coh}(\mathcal{X})$. If $\mathcal{X}$ is zero dimensional, then every object in $D_{coh}^{b}(\mathcal{X})$ is quasi-isomorphic to the direct sum of its cohomology sheaves. This is because the global section functor on $\mathcal{X}$ is exact since $\mathcal{X}$ is tame. The statement is clear in this case. If $\mathcal{X}$ has strictly positive dimension, then $\operatorname{Coh}(\mathcal{X})$ has an almost ample set by the observation above, and the statement follows from \cite[Proposition\ 3.3]{CS14}.
\end{proof}

Now we are in the position to prove Theorem \ref{1.1}.

\begin{proof}[Proof of Theorem \ref{1.1}]
For existence, we know from Lemma \ref{4.3} that there exists a bounded complex $K$ in $D_{coh}^{b}(\mathcal{X}\times_{k}\mathcal{Y})$ such that $F(\mathcal{F})\cong\Phi_{K}^{\mathcal{X\to Y}}(\mathcal{F})$ for any coherent sheaf $\mathcal{F}$ on $\mathcal{X}$ with zero dimensional support. By Lemma \ref{5.4}, the essential image of $\Phi_{K}^{\mathcal{X\to Y}}$ is contained in that of $F$. So the functor $G=F^{-1}\circ\Phi_{K}^{\mathcal{X\to Y}}$ is well-defined. In particular, we have $G(\mathcal{F})\cong\mathcal{F}$ for any coherent sheaf $\mathcal{F}$ on $\mathcal{X}$ with zero-dimensional support. By Lemma \ref{6.2}, we conclude that $G$ is a Fourier--Mukai autoequivalence on $D_{coh}^{b}(\mathcal{X})$. By construction, we have $F\cong\Phi_{K}^{\mathcal{X\to Y}}\circ G^{-1}.$ Since $G$ is a Fourier--Mukai transform, so is $G^{-1}$. This finishes the existence part of the proof. 

For uniqueness, suppose that $F\cong\Phi_{K}^{\mathcal{X\to Y}}$ for some bounded complex $K\in D_{coh}^{b}(\mathcal{X}\times_{k}\mathcal{Y})$. Consider the Fourier--Mukai transform $\Phi_{K\boxtimes\mathcal{O}_{\Delta_{\mathcal{X}}}}\colon D_{coh}^{b}(X\times_{k} X)\to D_{coh}^{b}(X\times_{k} Y)$. Apply this functor to the Postnikov system in \ref{eq4.1} gives us$$\Phi_{K\boxtimes\mathcal{O}_{\Delta}}(M_{n})\simeq\Phi_{K\boxtimes\mathcal{O}_{\Delta_{\mathcal{X}}}}(\mathcal{O}_{\Delta_{\mathcal{X}}})\oplus\Phi_{K\boxtimes\mathcal{O}_{\Delta_{\mathcal{X}}}}(\mathcal{F}_{n})[n]$$for all $n>0$. It is not hard to see that$$K\simeq\Phi_{K\boxtimes\mathcal{O}_{\Delta_{\mathcal{X}}}}(\mathcal{O}_{\Delta_{\mathcal{X}}}).$$So we have$$\Phi_{K\boxtimes\mathcal{O}_{\Delta}}(M_{n})\simeq\Phi_{K\boxtimes\mathcal{O}_{\Delta_{\mathcal{X}}}}(\mathcal{O}_{\Delta_{\mathcal{X}}})\oplus\Phi_{K\boxtimes\mathcal{O}_{\Delta_{\mathcal{X}}}}(\mathcal{F}_{n})[n]$$for all $n>0$. Note that any Fourier--Mukai transform is bounded. For a large enough $n$, we can find a small integer $j$ such that $K\simeq\tau_{\geq j}\Phi_{K\boxtimes\mathcal{O}_{\Delta}}(M_{n})$. But a simple computation shows that$$\Phi_{K\boxtimes\mathcal{O}_{\Delta_{\mathcal{X}}}}(\mathcal{E}_{i}\boxtimes\mathcal{F}_{i})\simeq\mathcal{E}_{i}\boxtimes F(\mathcal{F}_{i})$$for all $i$. This tells us that $\Phi_{K\boxtimes\mathcal{O}_{\Delta}}(M_{n})$ is the $n^{th}$ convolution of the complex in \ref{eq4.3}. By uniqueness of convolution, we have$$K\simeq\tau_{\geq j}\Phi_{K\boxtimes\mathcal{O}_{\Delta}}(M_{n})\simeq\tau_{\geq j}K_{n}$$for some fixed numbers $n$ and $j$. This proves that $K$ is unique.
\end{proof}

\section{Fully faithful functors and dimension}\label{7}

In this section, we provide an application of Theorem \ref{1.1}, which relates fully faithful functors and dimensions of the bounded derived categories of coherent sheaves on tame stacks. We first recall the notion of countable Rouquier dimension due to Pirozhkov.

\begin{definition}[{\cite[Definition\ 6]{Ola23}}]\label{7.1}
Let $\mathcal{T}$ be a triangulated category. The {\it countable Rouquier dimension} of $\mathcal{T}$, denoted by $\operatorname{CRdim}(\mathcal{T})$, is the smallest integer $n$ such that there exists a countable set $S$ of objects of $\mathcal{T}$ with $\mathcal{T}=\langle S\rangle_{n+1}$. 
\end{definition}

If require $S$ to be a singleton, then we recover the Rouquier dimension. See \cite[Definition\ 3.2]{Rouquier}. If $\mathcal{T}=D_{coh}^{b}(\mathcal{X})$ where $\mathcal{X}$ is an algebraic stack, we write $\operatorname{CRdim}(\mathcal{X})$ for $\operatorname{CRdim}(D_{coh}^{b}(\mathcal{X}))$. As remarked in \cite{Ola23}, the countable Rouquier dimension does not behave well for schemes over a finite field. Let $X$ be an algebraic variety over a finite field. Then $\operatorname{CRdim}(D_{coh}^{b}(X))=0$. Indeed, the derived category $D_{coh}^{b}(X)$ only has a countable number of objects up to isomorphism. The following lemma is clear from our definition.

\begin{lemma}[{\cite[Lemma\ 7]{Ola23}}]\label{7.2}
Let $F\colon\mathcal{T}\to\mathcal{S}$ be an exact functor between triangulated categories. If $F$ is essentially surjective, then $\operatorname{CRdim}(\mathcal{T})\geq\operatorname{CRdim}(\mathcal{S})$. 
\end{lemma}

Let $\mathcal{X}$ be a Noetherian algebraic stack with finite diagonal. In this section, we show that the countable Rouquier dimension and the Krull dimension of $\mathcal{X}$ agree if it is smooth and tame. We first establish the following two lemmas on the vanishing of Ext-groups and compositions of maps in $D_{coh}^{b}(\mathcal{X})$.

\begin{lemma}[cf.\ {\cite[\href{https://stacks.math.columbia.edu/tag/0FZ3}{Tag 0FZ3}]{stacks-project}}]\label{7.3}
Let $\mathcal{X}$ be a smooth tame algebraic stack of Krull dimension $d$ with quasi-finite diagonal. For any coherent sheaves $\mathcal{F}$ and $\mathcal{G}$ on $\mathcal{X}$, we have $\operatorname{Ext}_{\mathcal{O}_{\mathcal{X}}}^{i}(\mathcal{F},\mathcal{G})=0$ for all $i>d$.
\end{lemma}

\begin{proof}
By \cite[Theorem\ 7.1]{quasifinite}, we may choose a quasi-finite flat presentation $p\colon U\to\mathcal{X}$. Since $\mathcal{X}$ is smooth, it can be arranged that $U$ is also smooth. Since $p$ is quasi-finite, the Krull dimension of $\mathcal{X}$ and $U$ are the same. Since $p$ is flat, the natural map$$Lp^{*}\operatorname{RHom}_{\mathcal{O}_{\mathcal{X}}}(\mathcal{F},\mathcal{G})\to\operatorname{RHom}_{\mathcal{O}_{U}}(\pi^{*}\mathcal{F},\pi^{*}\mathcal{F})$$is an isomorphism by \cite[\href{https://stacks.math.columbia.edu/tag/0GM7}{Tag 0GM7}]{stacks-project}. It follows that$$p^{*}\sheafext^{i}(\mathcal{F},\mathcal{G})\cong\sheafext^{i}(p^{*}\mathcal{F},p^{*}\mathcal{G})$$for all $i$. In this case, we have the following spectral sequence$$H^{a}(U, p^{*}\sheafext_{\mathcal{O}_{\mathcal{X}}}^{b}(\mathcal{F},\mathcal{G}))\cong H^{a}(U, \sheafext_{\mathcal{O}_{U}}^{b}(p^{*}\mathcal{F},p^{*}\mathcal{G}))\implies\operatorname{Ext}_{\mathcal{O}_{U}}^{a+b}(p^{*}\mathcal{F},p^{*}\mathcal{G}).$$By \cite[\href{https://stacks.math.columbia.edu/tag/0FZ3}{Tag 0FZ3}]{stacks-project}, we know that $\operatorname{Ext}_{\mathcal{O}_{U}}^{p+q}(p^{*}\mathcal{F},p^{*}\mathcal{G})=0$ if $a>d-b$. It follows that $H^{a}(U, p^{*}\sheafext_{\mathcal{O}_{\mathcal{X}}}^{b}(\mathcal{F},\mathcal{G}))=0$ whenever $a>d-b$. Since $p$ is faithfully flat, we also have $H^{a}(\mathcal{X}, \sheafext_{\mathcal{O}_{\mathcal{X}}}^{b}(\mathcal{F},\mathcal{G}))=0$ whenever $a>d-b$. This finishes the proof.
\end{proof}

\begin{lemma}\label{7.4}
Let $\mathcal{X}$ be a smooth tame algebraic stack of Krull dimension $d$ with quasi-finite diagonal. Let $K_{0}\to K_{1}\to\cdots\to K_{d}\to K_{d+1}$ be some morphisms in $D_{coh}^{b}(\mathcal{X})$ whose induced maps on cohomology sheaves vanish. Then the composition $K_{0}\to K_{d+1}$ is the zero morphism in $D_{coh}^{b}(\mathcal{X})$. 
\end{lemma}

\begin{proof}
We argue as in the proof of \cite[Theorem\ 2]{Ola23}. Since $\mathcal{X}$ is smooth and tame, $D_{coh}^{b}(\mathcal{X})$ is a full subcategory of $D(QCoh(\mathcal{X}))$ by \cite[Theorem\ 2.1(2)]{Hal22}. Note that $\operatorname{QCoh}(\mathcal{X})$ has enough injectives. By \cite[Proposition\ 1]{Ola23}, we obtain a filtration$$\operatorname{Hom}(K_{0},K_{d+1})=F^{0}\supset F^{1}\supset F^{2}\supset\cdots.$$By Lemma \ref{7.3} and \cite[Proposition\ 1(2)]{Ola23}, we have $F^{d+1}=F^{d+j}$ for all $j\geq 1$. But we know that $F^{a}=0$ for a large enough $a$. It follows that $F^{d+1}=0$. By \cite[Proposition\ 1(4)]{Ola23}, the map $K_{i}\to K_{i+1}$ is in $F^{1}\operatorname{Hom}(K_{i},K_{i+1})$. Therefore the composition $K_{0}\to K_{d+1}$ is in $F^{d+1}\operatorname{Hom}(K_{0},K_{1})$, which is 0. 
\end{proof}

\begin{theorem}[cf.\ {\cite[Theorem\ 2.17]{BF12}}\ and\ {\cite[Proposition\ 9]{Ola23}}]\label{7.5}
Let $\mathcal{X}$ be a smooth tame algebraic stack of Krull dimension $d$ of finite type with finite diagonal over an uncountable field $k$. Then $\operatorname{CRdim}(\mathcal{X})=\dim(\mathcal{X})$.
\end{theorem}

\begin{proof}
We argue as in the proof of \cite[Proposition\ 9]{Ola23}. By assumption, $\mathcal{X}$ admits a coarse moduli space $\pi\colon\mathcal{X}\to X$. We prove the statement by showing the following chain of inequalities$$\operatorname{CRdim}(\mathcal{X})\geq\operatorname{CRdim}(X)\geq\dim(X)=\dim(\mathcal{X})\geq\operatorname{CRdim}(\mathcal{X}).$$The equality in the middle holds because $\pi$ is a universal homomorphism. Since $\mathcal{X}$ is tame, $\pi_{*}$ is also a good moduli space. In particular, $\pi_{*}$ is exact and the unit of the adjunction $id\to\pi_{*}\pi^{*}$ is an isomorphism. This tells us that $R\pi_{*}\colon D_{coh}^{b}(\mathcal{X})\to D_{coh}^{b}(X)$ is essentially surjective. By Lemma \ref{7.2}, we have $\operatorname{CRdim}(\mathcal{X})\geq\operatorname{CRdim}(X)$. Therefore we may assume that $\mathcal{X}$ is an algebraic space. By \cite[Proposition\ 2.1.4]{olander_2022}, we have $\operatorname{CRdim}(X)\geq\dim(X)$. It follows that$$\operatorname{CRdim}(\mathcal{X})\geq\operatorname{CRdim}(X)\geq\dim(X)=\dim(\mathcal{X}).$$

It remains to show $\dim(\mathcal{X})\geq\operatorname{CRdim}(X)$. In section 2, we produced a countable set $\{\mathcal{F}_{i}\}_{i\geq 0}$ of coherent sheaves on $\mathcal{X}$. We claim that $D_{coh}^{b}(\mathcal{X})=\langle\{\mathcal{F}_{i}\}_{i\geq 0}\rangle_{d+1}$. This would imply that $\operatorname{CRdim}(\mathcal{X})\leq\dim(\mathcal{X})$. To see this, we argue as in the proof of \cite[Theorem\ 4]{Ola23}. By Lemma \ref{2.2}, any bounded complex $K_{0}$ in $D_{coh}^{b}(\mathcal{X})$ admits a map $\sum_{n\in\mathbb{Z}}\mathcal{F}_{n}^{\oplus m_{i}}[M_{i}]\to K_{0}$ that is surjective on the cohomology sheaves for some $m_{i}$ and $M_{i}$. Let $K_{1}$ be the cone of the map above. By construction, $K_{1}$ is contained in $\langle\{\mathcal{F}_{n}\}_{n\in\mathbb{Z}}\rangle_{1}$. Repeat the above process, we get a sequence$$K_{0}\to K_{1}\to\cdots\to K_{d}\to K_{d+1}.$$ By Lemma \ref{7.4}, the composition $K_{0}\to K_{d+1}$ is 0.

Now we show that the cone of $K_{0}\to K_{i}$ is in $\langle\{\mathcal{F}_{n}\}_{n\in\mathbb{Z}}\rangle_{i}$. We do so by induction on $1\leq j\leq d+1$. The base case is shown above. For any $j>1$, let $C_{j}$ be the mapping cone of $K\to K_{j}$ and let $D_{j}$ be the mapping cone of the map $K_{j}\to K_{j+1}$. By the octahedral axiom, there is a distinguished triangle$$C_{j}\to C_{j+1}\to D_{j}\to C_{j}[1]$$for any $1<j\leq d+1$. We know that $D_{j}\in\langle\{\mathcal{F}_{n}\}_{n\in\mathbb{Z}}\rangle_{1}$. If $C_{j}\in\langle\{\mathcal{F}_{n}\}_{n\in\mathbb{Z}}\rangle_{j}$ for a fixed $j$, then $C_{j+1}\in\langle\{\mathcal{F}_{n}\}_{n\in\mathbb{Z}}\rangle_{j}\star\langle\{\mathcal{F}_{n}\}_{n\in\mathbb{Z}}\rangle_{1}=\langle\{\mathcal{F}_{n}\}_{n\in\mathbb{Z}}\rangle_{j+1}$. By induction on $j$, the claim is true. Since the map $K_{0}\to K_{d+1}$ is zero, we have $C_{d+1}\cong K_{0}[1]\oplus K_{d+1}\in\langle\{\mathcal{F}_{n}\}_{n\in\mathbb{Z}}\rangle_{d+1}$. It follows that $K=K_{0}\in\langle\{\mathcal{F}_{n}\}_{n\in\mathbb{Z}}\rangle_{d+1}$ as desired. 
\end{proof}

We say an algebraic stack $\mathcal{X}$ has the 1-resolution property if there exists a vector bundle $\mathcal{V}$ on $\mathcal{X}$ such that any coherent sheaf on $\mathcal{F}$ admits a subjective map from $\mathcal{V}^{\otimes n}$ for some $n>0$. For example, any quasi-affine scheme has the 1-resolution property since the structure sheaf is ample. Normal algebraic stacks with the 1-resolution property have been classified in \cite{DHM22}. Using the same argument for Theorem \ref{7.5}, one could show that the Rouquier dimension agrees with the Krull dimension for any regular tame algebraic stack with the 1-resolution property, extending a previous result of Olander for quasi-affine schemes \cite[Corollary\ 5]{Ola23}.

Now we are ready to prove the main theorem of this section.

\begin{proof}[Proof of Theorem \ref{1.3}]
By Theorem \ref{1.1}, $F\cong\Phi_{K}^{\mathcal{X}\to\mathcal{Y}}$ for some $K$ in $D_{coh}^{b}(\mathcal{X}\times_{k}\mathcal{Y})$. Let $k\to l$ be any uncountable field extension. Base changing from $k$ to $l$ yields an $l$-linear exact functor $F_{l}=\Phi_{E_{l}}^{\mathcal{X}_{l}\to\mathcal{Y}_{l}}$ that is also fully faithful. Therefore we may assume $k$ is uncountable. Since $F$ is fully faithful, it has a right adjoint $R$ by \cite[\href{https://stacks.math.columbia.edu/tag/0FYL}{Tag 0FYL}]{stacks-project}. In particular, $R$ is essentially surjective. It follows from Lemma \ref{7.2} and Proposition \ref{7.5} that
\begin{equation*}
    \dim(\mathcal{X})=\operatorname{CRdim}(\mathcal{X})\leq\operatorname{CRdim}(\mathcal{Y})=\dim(\mathcal{Y}).\qedhere
\end{equation*}
\end{proof}

Let $X$ be a normal proper scheme over a perfect field with at worst finite linearly reductive quotient singularities. By \cite[Theorem\ 1.10]{Sat12}, there is a unique smooth tame algebraic stack $X^{can}$ such that $X^{can}\to X$ is a coarse moduli space. In particular, $X^{can}$ is proper. We conclude this article with the following special case of Theorem \ref{1.3}.

\begin{corollary}\label{7.6}
Let $X, Y$ be normal proper schemes over a perfect field with at worst finite linearly reductive quotient singularities. Let $X^{can},Y^{can}$ be the associated smooth tame algebraic stacks. If $D_{coh}^{b}(\mathcal{X}^{can})\cong D_{coh}^{b}(\mathcal{Y}^{can})$, then $\dim(X)=\dim(Y)$.
\end{corollary}

\begin{proof}
By Theorem \ref{1.3}, we have
\begin{equation*}
\dim(X)=\dim(\mathcal{X}^{can})=\dim(\mathcal{Y}^{can})=\dim(Y).\qedhere
\end{equation*}
\end{proof}

\appendix

\section{Functors between categories of quasi-coherent sheaves}\label{A}

In this appendix, we establish some technical results for quasi-coherent sheaves and Fourier--Mukai functors that we used to prove Theorem \ref{1.2}. The content is mostly an exposition of the results in \cite[\href{https://stacks.math.columbia.edu/tag/0FZA}{Tag 0FZA}]{stacks-project}. Let $R$ be a ring. Let $\mathcal{X},\mathcal{Y}$ be algebraic stacks over $R$ with $X$ quasi-compact and quasi-separated. Let $\mathcal{K}$ be an object in $\operatorname{QCoh}(\mathcal{X}\times_{R}\mathcal{Y})$. Then there exists an $R$-linear functor$$F_{\mathcal{K}}\colon\operatorname{QCoh}(\mathcal{X})\to\operatorname{QCoh}(\mathcal{Y})$$sending a quasi-coherent sheaf $\mathcal{F}$ on $\mathcal{X}$ to $p_{2,*}(p_{1}^{*}\mathcal{F}\otimes_{\mathcal{O}_{\mathcal{X}\times_{R}\mathcal{Y}}}\mathcal{K})$. Since $X$ is quasi-compact and quasi-separated, this functor is well-defined. Moreover, the functor $F_{\mathcal{K}}$ is also $R$-linear and preserves arbitrary direct sums.

\begin{lemma}[cf.\ {\cite[\href{https://stacks.math.columbia.edu/tag/0FZD}{Tag 0FZD}]{stacks-project}}]\label{A.1}
Let $R$ be a ring. Let $\mathcal{X},\mathcal{Y}$ be algebraic stacks over $R$. If $\mathcal{X}$ is an affine scheme, then there exists an equivalence of categories between
\begin{enumerate}[topsep=0pt,noitemsep,label=\normalfont(\arabic*)]
    \item the category of right exact $R$-linear functors $F\colon\operatorname{QCoh}(\mathcal{X})\to\operatorname{QCoh}(\mathcal{Y})$ that commute with arbitrary direct sums, and
    \item the category $\operatorname{QCoh}(\mathcal{X}\times_{R}\mathcal{Y})$.
\end{enumerate}
\end{lemma}

\begin{proof}
Suppose $\mathcal{X}=\operatorname{Spec}A$ for some $R$-algebra $A$. The proof is essentially the same as that of \cite[\href{https://stacks.math.columbia.edu/tag/0FZD}{Tag 0FZD}]{stacks-project}. Let $\mathcal{K}$ be a quasi-coherent sheaf on $\mathcal{X}\times\mathcal{Y}$. Then the associated Fourier--Mukai transform $F_{\mathcal{K}}$ is an $R$-linear functor that commutes with arbitrary direct sums. Since $\mathcal{X}$ is affine, the projection $p_{2}\colon\mathcal{X}\times_{R}\mathcal{Y}\to\mathcal{Y}$ is affine. Therefore $p_{2}$ is exact, and thus $F_{\mathcal{K}}$ is right exact.

Let $F\colon\operatorname{QCoh}(\mathcal{X})\to\operatorname{QCoh}(\mathcal{Y})$ be an $R$-linear functor that commute with arbitrary direct sum. Set $\mathcal{G}=F(\mathcal{O}_{\mathcal{X}})$. Then the functor $F$ induces a ring homomorphism $A\to\operatorname{End}_{\mathcal{O}_{\mathcal{Y}}}(\mathcal{G})$ sending $a$ to $F(a\cdot id)$. This gives us an $A\otimes_{R}\mathcal{Y}$-module structure on $\mathcal{G}$. Note that $A\otimes_{R}\mathcal{Y}\cong p_{2,*}\mathcal{O}_{\mathcal{X}\times_{R}\mathcal{Y}}$. This tells us that $\mathcal{G}$ is a quasi-coherent $\mathcal{O}_{\mathcal{X}\times_{R}\mathcal{Y}}$-module. Since $p_{2}$ is affine, there exists a unique quasi-coherent sheaf $\mathcal{K}^\prime$ on $\mathcal{X}\times_{R}\mathcal{Y}$ such that $\mathcal{G}\cong p_{2,*}\mathcal{K}^\prime$. Consider the associated Fourier--Mukai transform $F_{\mathcal{K}^\prime}$. We see that $F$ and $F_{\mathcal{K}^\prime}$ both send $\mathcal{O}_{\mathcal{X}}$ to $\mathcal{O}_{\mathcal{Y}}$ and are compatible with the action of $A$ and $\mathcal{O}_{\mathcal{Y}}$. It follows from \cite[\href{https://stacks.math.columbia.edu/tag/0GNQ}{Tag 0GNQ}]{stacks-project} that $F\cong F_{\mathcal{K}^\prime}$. This finishes the proof.
\end{proof}

If we restrict to the category of exact $R$-linear functors $F\colon\operatorname{QCoh}(\mathcal{X})\to\operatorname{QCoh}(\mathcal{Y})$ that commute with arbitrary direct sums, then the essential image of this category under the equivalence in Lemma \ref{A.1} is precisely the category of quasi-coherent sheaves on $\mathcal{X}\times_{R}\mathcal{Y}$ that are flat over $\mathcal{X}$. 

\begin{corollary}\label{A.2}
In the situation of Lemma \ref{A.1}, there exists an equivalence of categories between
\begin{enumerate}[topsep=0pt,noitemsep,label=\normalfont(\arabic*)]
    \item the category of exact $R$-linear functors $F\colon\operatorname{QCoh}(\mathcal{X})\to\operatorname{QCoh}(\mathcal{Y})$, and
    \item the full subcategory of $\operatorname{QCoh}(\mathcal{X}\times_{R}\mathcal{Y})$ consisting of quasi-coherent sheaves flat over $\mathcal{X}$.
\end{enumerate}
\end{corollary}
\begin{proof}
Let $F\colon\operatorname{QCoh}(\mathcal{X})\to\operatorname{QCoh}(\mathcal{Y})$ be an $R$-linear exact functor that commute with arbitrary direct sums. Lemma \ref{A.1} gives us a quasi-coherent sheaf $\mathcal{K}$ on $\mathcal{O}_{\mathcal{X}\times_{R}\mathcal{Y}}$ such that $F\cong F_{\mathcal{K}}$. We know that $F$ is exact. Since $\mathcal{X}$ is affine, $p_{2,*}$ is exact. It follows that $\mathcal{K}$ must be flat over $\mathcal{X}$.
\end{proof}

\begin{lemma}[cf.\ {\cite[\href{https://stacks.math.columbia.edu/tag/0GPA}{Tag 0GPA}]{stacks-project}}]\label{A.3}
Let $R$ be a ring. Let $\mathcal{X},\mathcal{Y}$ be algebraic stacks over $R$ such that $\mathcal{X}$ is quasi-compact with affine diagonal. Let $F\colon\operatorname{QCoh}(\mathcal{X})\to\operatorname{QCoh}(\mathcal{Y})$ be an $R$-linear right exact functor that commutes with arbitrary direct sums. Then there exists an object $\mathcal{K}$ in $\operatorname{QCoh}(\mathcal{X}\times_{R}\mathcal{Y})$ and a natural transformation $t\colon F\to F_{\mathcal{K}}$ such that $t\colon F\circ f_{*}\to F_{\mathcal{K}}\circ f_{*}$ is an isomorphism for any smooth presentation $f\colon U\to\mathcal{X}$ with $U$ an affine scheme. In particular, if $F$ is exact, it can be arranged that $t$ is a natural isomorphism and $\mathcal{K}$ is flat over $\mathcal{X}$.
\end{lemma}

\begin{proof}
Let $f\colon U\to\mathcal{X}$ be a smooth presentation of $\mathcal{X}$ where $U$ is an affine scheme. Set $\mathcal{R}=U\times_{\mathcal{X}}U$. Since $\mathcal{X}$ has affine diagonal, $R$ is also an affine scheme. Then we have the following cartesian diagram
\begin{equation*}
    \begin{tikzcd}
        \mathcal{R}\arrow[r,"s"]\arrow[d,"t"] & U\arrow[d,"f"]\\
        U\arrow[r,"f"] & \mathcal{X}
    \end{tikzcd}
\end{equation*}
Note that the maps in the diagram are all smooth and affine. We see that the quintuple$$G=(U\times_{R}\mathcal{Y},\mathcal{R}\times_{R}\mathcal{Y},s\times_{R}id_{\mathcal{Y}},t\times_{R}id_{\mathcal{Y}},c\times_{R}id_{\mathcal{Y}}),$$where $c\colon R\times_{s,\times,t}R\to R$ is the canonical projection, forms a smooth groupoid in affine schemes over $\mathcal{Y}$ such that the quotient stack $[U/R]$ is isomorphic to $\mathcal{X}\times_{R}\mathcal{Y}$.

Define $F^\prime=F\circ f_{*}$. It is clear that $F^\prime$ is an $R$-linear functor that commutes with arbitrary direct sums. Since $f$ is affine, the composition $F^\prime$ is right exact. By Lemma \ref{A.1}, there exists a quasi-coherent sheaf $\mathcal{K}^\prime$ on $U\times_{R}\mathcal{Y}$ such that $F^\prime\cong F_{\mathcal{K}^\prime}$. Consider the functors $F^\prime\circ s_{*}$ and $F^\prime\circ t_{*}$. It follows from \cite[\href{https://stacks.math.columbia.edu/tag/0FZF}{Tag 0FZF}]{stacks-project} that these two functors correspond to the quasi-coherent sheave $(s\times_{R}id_{\mathcal{Y}})^{*}\mathcal{K}^\prime$ and $(t\times_{R}id_{\mathcal{Y}})^{*}\mathcal{K}^\prime$, respectively. By \cite[\href{https://stacks.math.columbia.edu/tag/0GPA}{Tag 0GPA}]{stacks-project}, we have$$F^\prime\circ s_{*}\cong F_{\mathcal{K}^\prime}\cong F^\prime\circ t_{*}.$$By Lemma \ref{A.1}, there exists an isomorphism$$\alpha\colon(s\times_{R}id_{\mathcal{Y}})^{*}\mathcal{K}^\prime\to(t\times_{R}id_{\mathcal{Y}})^{*}\mathcal{K}^\prime.$$This defines a quasi-coherent sheaf $(\mathcal{K}^\prime,\alpha)$ on the groupoid $G$. By \cite[\href{https://stacks.math.columbia.edu/tag/06WT}{Tag 06WT}]{stacks-project}, there exists a unique quasi-coherent sheaf $\mathcal{K}$ on $\mathcal{X}\times_{R}\mathcal{Y}$ such that $(f\times_{R}id_{\mathcal{Y}})^{*}\mathcal{K}\cong\mathcal{K}^\prime$. 

Now we construct the natural transformation $t$. Let $(U/\mathcal{X})^{d}$ be the fiber product of $U$ with itself over $\mathcal{X}$ with $d+1$ factors. Let $f^{d}$ be the structure morphism onto $\mathcal{X}$ and let $F^{d}=f^{d}\times_{R}id_{\mathcal{Y}}$. In the case where $d=0$, we set $f^{0}=f$. Since $\mathcal{X}$ has affine diagonal, the product $(U/\mathcal{X})^{d}$ is an affine scheme, and the map $f^{d}$ is an affine morphism for all $d$. Let $\mathcal{K}^{d}$ be the quasi-coherent sheaf on $((U/\mathcal{X})^{d})\times_{R}\mathcal{Y}$ corresponding to $F\circ f_{*}^{d}$. By Lemma \ref{A.1}, we see that $\mathcal{K}^{d}\cong F^{d,*}\mathcal{K}$, where $\mathcal{K}$ is the quasi-coherent sheaf on $\mathcal{X}\times_{R}\mathcal{Y}$ obtained in the last paragraph. Let $\mathcal{F}$ be a quasi-coherent sheaf on $\mathcal{X}$. Set $\mathcal{F}^{d}=f^{d,*}\mathcal{F}$. Observe that$$F(f_{*}^{d}\mathcal{F}^{d})\cong p_{2,*}^{d}F^{d,*}(p_{1}^{*}\mathcal{F}\otimes\mathcal{K})$$where $p_{2}^{d}$ is the canonical projection $((U/\mathcal{X})^{d})\times_{R}\mathcal{Y}\to\mathcal{Y}$. Then we obtain the following short exact sequence$$0\longrightarrow\mathcal{F}\longrightarrow f_{0,*}\mathcal{F}^{0}\longrightarrow f_{1,*}\mathcal{F}^{1}$$by truncating the Cech complex. Since $F$ is right exact, we have the following exact sequence$$F(\mathcal{F})\longrightarrow F(f_{0,*}\mathcal{F}^{0})\longrightarrow F(f_{1,*}\mathcal{F}^{1}).$$On the other hand, $p_{1}^{*}\mathcal{F}\otimes\mathcal{K}$ is a quasi-coherent sheaf on $\mathcal{X}\times_{R}\mathcal{Y}$. We also have the following short exact sequence$$0\longrightarrow F_\mathcal{K}(\mathcal{F})\longrightarrow p_{2,*}^{0}F^{0,*}(p_{1}^{*}\mathcal{F}\otimes\mathcal{K})\longrightarrow p_{2,*}^{1}F^{1,*}(p_{1}^{*}\mathcal{F}\otimes\mathcal{K})\longrightarrow\cdots$$by truncating the relative Cech complex along $p_{2}$. This gives us the following diagram
\begin{equation*}
    \begin{tikzcd}
       & F(\mathcal{F})\arrow[r]\arrow[d,dashed] & F(f_{*}^{0}\mathcal{F}^{0})\arrow[r]\arrow[d] & F(f_{*}^{1}\mathcal{F}^{1})\arrow[d]\\
        0\arrow[r] & F_\mathcal{K}(\mathcal{F})\arrow[r] & p_{2,*}^{0}F^{0,*}(p_{1}^{*}\mathcal{F}\otimes\mathcal{K}) \arrow[r] & p_{2,*}^{1}F^{1,*}(p_{1}^{*}\mathcal{F}\otimes\mathcal{K}).
    \end{tikzcd}
\end{equation*}
We have shown that the vertical solid arrows are isomorphisms. It follows that the dashed arrow can be filled in by a map $t_{\mathcal{F}}$. It is clear that the formation of $t_{\mathcal{F}}$ is functorial on $\mathcal{F}$. This gives rise to the desired natural transformation $t$. In the special case where $F$ is exact, the map $F(\mathcal{F})\to F(f_{*}^{0}\mathcal{F}^{0})$ would be injective as well. Therefore the map $t_{\mathcal{F}}$ must be an isomorphism for all $\mathcal{F}$. Hence, $t$ is an isomorphism. The flatness of $\mathcal{K}$ follows from Corollary \ref{A.2}.\qedhere

\end{proof}

We conclude this section with the following variant of Theorem \ref{1.2} for quasi-coherent sheaves.

\begin{proposition}[cf.\ {\cite[\href{https://stacks.math.columbia.edu/tag/0FZH}{Tag 0FZH}]{stacks-project}}]\label{A.4}
Let $R$ be a ring. Let $\mathcal{X},\mathcal{Y}$ be algebraic stacks over $R$. If $\mathcal{X}$ is quasi-compact with affine diagonal, then there exists an equivalence of categories between
\begin{enumerate}[topsep=0pt,noitemsep,label=\normalfont(\arabic*)]
    \item\label{3.3(1)} the category of $R$-linear exact functors $F\colon\operatorname{QCoh}(\mathcal{X})\to\operatorname{QCoh}(\mathcal{Y})$ that commute with arbitrary direct sum, and
    \item\label{3.3(2)} the full subcategory of $\operatorname{QCoh}(\mathcal{X}\times_{R}\mathcal{Y})$ consisting of $\mathcal{K}$ such that
    \begin{enumerate}[topsep=0pt,noitemsep,label=\normalfont(\alph*)]
        \item $\mathcal{K}$ is flat over $\mathcal{X}$, and
        \item for all quasi-coherent sheaf $\mathcal{F}$ on $\mathcal{X}$ we have $R^{i}p_{2,*}(p_{1}^{*}\mathcal{F}\otimes_{\mathcal{O}_{\mathcal{X}\times_{R}\mathcal{Y}}}\mathcal{K})=0$ for all $i>0$
    \end{enumerate}
\end{enumerate}
sending a quasi-coherent sheaf $\mathcal{K}$ to the functor $F_{\mathcal{K}}$.
\end{proposition}

\begin{proof}
Let $\mathcal{K}$ be a quasi-coherent sheaf on $\mathcal{X}\times_{R}\mathcal{Y}$ as in \ref{3.3(2)}. It is clear that $F_{\mathcal{K}}$ is exact. It suffices to find a quasi-inverse to this construction. Let $F$ be an exact functor as in \ref{3.3(1)}. By Lemma \ref{A.3}, there exists a quasi-coherent sheaf $\mathcal{K}$ on $\mathcal{X}\times_{R}\mathcal{Y}$ that is flat over $\mathcal{X}$ with a natural isomorphism $t\colon F\to F_{\mathcal{K}}$. It suffices to show the vanishing of the higher direct images. We use the notations introduced in the proof of Lemma \ref{A.3}. Since $F$ is exact, we have a quasi-ismorphism 
\begin{equation*}
    \begin{tikzcd}
        0 \arrow[r] & F(\mathcal{F})\arrow[r]\arrow[d,"t_{\mathcal{F}}"] & F(f_{0,*}\mathcal{F}^{0})\arrow[r]\arrow[d] & F(f_{1,*}\mathcal{F}^{1})\arrow[r]\arrow[d]&\cdots\\
        0 \arrow[r] & F_\mathcal{K}(\mathcal{F})\arrow[r] & p_{2,*}^{0}F^{0,*}(p_{1}^{*}\mathcal{F}\otimes\mathcal{K})\arrow[r] & p_{2,*}^{1}F^{1,*}(p_{1}^{*}\mathcal{F}\otimes\mathcal{K})\arrow[r]&\cdots
    \end{tikzcd}
\end{equation*}
for any quasi-coherent sheaf $\mathcal{F}$ on $\mathcal{X}$. Note that $Rp_{2,*}(p_{1}^{*}\mathcal{F}\otimes\mathcal{K})$ is precisely given by the cohomology of the bottom row. But the top row is exact everywhere except at $F(\mathcal{F})$. It follows that $R^{i}p_{2,*}(p_{1}^{*}\mathcal{F}\otimes\mathcal{K})=0$ for all $i>0$. This finishes the proof.
\end{proof}

\bibliographystyle{amsalpha}
\bibliography{orlovbib.bib}

\end{document}